\documentclass[11pt]{article}
\pdfoutput=1
\usepackage[left=1in,right=1.2in]{geometry}
\usepackage[colorlinks=true,linkcolor=blue]{hyperref}
\usepackage{amsmath,amsthm,amssymb}
\usepackage{color} 
\usepackage{graphicx}
\graphicspath{{./figures/}}

\numberwithin{equation}{section}

\usepackage{tikz} 
\usetikzlibrary{decorations.pathmorphing} 
\usetikzlibrary{decorations.markings} 

\usepackage[color=yellow!35]{todonotes}

\newcounter{todocounter}

\newcommand{\overarrowsplus}{\raisebox{-.6ex}{\tikz[-{To[]},scale=.3]{	\draw (0,1) -- (1,0);\draw[white,-,line width= 3pt] (1,1) -- (0,0);	\draw (1,1) -- (0,0);}}}
\newcommand{\overarrowsminus}{\raisebox{-.6ex}{\tikz[-{To[]},scale=.3]{	\draw (1,1) -- (0,0);\draw[white,-,line width= 3pt] (0,1) -- (1,0);	\draw (0,1) -- (1,0);}}}
\newcommand{\downcurvearrowleft}{\raisebox{1ex}{\scalebox{-1}{$\curvearrowright$}}}
\newcommand{\downcurvearrowright}{\raisebox{1ex}{\scalebox{-1}{$\curvearrowleft$}}}

\tikzset{%
	leftCrossing/.pic={code={%
			\tikzset{->,very thick}
			\draw (0,1) to[out=-90,in=90]+(1,-1);
			\draw[-,white,line width=4pt,postaction={draw,black,thick,->}] (1,1) to[out=-90,in=90]+(-1,-1);}},
	rightCrossing/.pic={code={%
			\tikzset{->,very thick}
			\draw (1,1) to[out=-90,in=90]+(-1,-1);
			\draw[-,white,line width=4pt,postaction={draw,black,thick,->}] (0,1) to[out=-90,in=90]+(1,-1);}}
}

\newcommand{\be}{\begin{equation}}
\newcommand{\ee}{\end{equation}}

\newcommand{\CF}{{\mathcal F}}

\newcommand{\CV}{{\mathcal V}}

\newcommand{\cat}{\mathcal{C}}
\newcommand{\Rib}{\mathcal{R}}
\newcommand{\C}{\ensuremath{\mathbb{C}} }
\newcommand{\Z}{\ensuremath{\mathbb{Z}} }

\newcommand{\N}{\ensuremath{\mathbb{N}} }
\newcommand{\W}{\ensuremath{\mathbb{W}} }
\newcommand{\E}{\ensuremath{\mathbb{E}} }
\newcommand{\rtu}{\ensuremath{{\zeta_{2r}}}}
\newcommand{\Fq}{{\mathbb K_r}} 
\newcommand{\Uqg}{\ensuremath{\mathcal U}}
\newcommand{\UqgH}{\ensuremath{\Uqg^{H}}}
\newcommand{\Ubar}{\ensuremath{\overline{\Uqg}^{H}_\rtu}}
\newcommand{\unit}{\ensuremath{{\mathrm{1}\mkern-4mu{\mathchoice{}{}{\mskip-0.5mu}{\mskip-1mu}}\mathrm{l}}}}
\newcommand{\End}{\operatorname{End}}
\newcommand{\Hom}{\operatorname{Hom}}
\newcommand{\Id}{\operatorname{Id}}
\newcommand{\coev}{\stackrel{\longrightarrow}{\operatorname{coev}}}
\newcommand{\ev}{\stackrel{\longrightarrow}{\operatorname{ev}}}
\newcommand{\tev}{\stackrel{\longleftarrow}{\operatorname{ev}}}
\newcommand{\tcoev}{\stackrel{\longleftarrow}{\operatorname{coev}}}
\newcommand{\mathsmall}[1]{\mbox{\small$#1$}}
\newcommand{\qdim}{\operatorname{qdim}} 
\newcommand{\md}{\operatorname{\mathsf{d}}}

\newcommand{\pic}[2]{
	\setlength{\unitlength}{#1}
	{\begin{array}{c} \hspace{-1.3mm}
			\raisebox{-4pt}{#2}
			\hspace{-1.9mm}\end{array}}}

\newcommand{\drawTr}{
	\qbezier(3, 3)(3, 0)(6, 0)
	\qbezier(6, 0)(9, 0)(9, 4)
	\put(9,4){\vector(0,1)3}
	\multiput(0,3)(6,0){2}{\line(0,1){5}}
	\multiput(0,3)(0,5){2}{\line(1,0){6}}
	\qbezier(3, 8)(3, 11)(6, 11)
	\qbezier(6, 11)(9, 11)(9, 7)
}

\newcommand{\drawQDim}{
	\qbezier(0, 3)(0, 0)(3, 0)
	\qbezier(3, 0)(6, 0)(6, 3)
	\put(6,3.5){\vector(0,1)1}
	\qbezier(0, 3)(0, 7)(3, 7)
	\qbezier(3, 7)(6, 7)(6, 4)
}

\newtheorem{counter}{Counter}[section]
\newtheorem{thm:ADO-q-hol}[counter]{Theorem}
\newtheorem{def:q-hol}[counter]{Defintion}
\newtheorem{def:q-proper-hypergeo}[counter]{Definition}
\newtheorem{thm:q-hol-closure}[counter]{Theorem}
\newtheorem{cor:tensor-contraction}[counter]{Corollary}

\newtheorem{theorem}[counter]{Theorem}
\newtheorem{definition}[counter]{Definition}
\newtheorem{remark}[counter]{Remark}
\newtheorem{prop}[counter]{Proposition}
\newtheorem{lemma}[counter]{Lemma}
\newtheorem{cor}[counter]{Corollary}

\newtheorem{lemma:W1-single-recursion}[counter]{Lemma}
\newtheorem{lemma:generators}[counter]{Lemma}

\begin{document}
	\title{The ADO Invariants are a q-Holonomic Family}
	
	\author{\hspace{-.4in}Jennifer Brown$^a$\,, Tudor Dimofte$^a$\,, Stavros Garoufalidis$^{b,c}$\,, and Nathan Geer$^d$\\ \\
\hspace{-.4in}	 {\small $^a$Department of Mathematics and QMAP, UC Davis, 1 Shields Ave, Davis, CA 95616, USA}\\
\hspace{-.4in}	 {\small $^b$Max Planck Institute for Mathematics, Vivatsgasse 7, 53111 Bonn, Germany} \\
\hspace{-.4in}	 {\small $^c$International Center for Mathematics, Dept. of Math., Southern Univ. of Science and Technology, Shenzhen, China} \\
\hspace{-.4in}	 {\small $^d$Department of Mathematics \& Statistics, Utah State University, Logan, Utah 84322, USA} \\ 
\hspace{-.4in}	 {\small \texttt{brown@math.ucdavis.edu}\quad \texttt{tudor@math.ucdavis.edu} \quad \texttt{stavros@mpim-bonn.mpg.de}\quad \texttt{nathan.geer@gmail.com}}
	  }

	\maketitle
	
	\begin{abstract}
		We investigate the $q$-holonomic properties of a class of link invariants based on quantum group representations with vanishing quantum dimensions, motivated by the search for the invariants' realization in physics. Some of the best known invariants of this type, constructed from `typical' representations of the unrolled quantum group $\mathcal U^H_{\zeta_{2r}}(\mathfrak{sl}_2)$ at a $2r$-th root of unity, were introduced by Akutsu-Deguchi-Ohtsuki (ADO). We prove that the ADO invariants for $r\geq 2$ are a $q$-holonomic \emph{family}, implying in particular that they satisfy recursion relations that are independent of $r$. In the case of a knot, we prove that the $q$-holonomic recursion ideal of the ADO invariants is contained in the recursion ideal of the colored Jones polynomials, the subject of the celebrated AJ Conjecture. (Combined with a recent result of S. Willetts, this establishes an isomorphism of the ADO and Jones recursion ideals. Our results also confirm a recent physically-motivated conjecture of Gukov-Hsin-Nakajima-Park-Pei-Sopenko.)
			\end{abstract}
	
	\tableofcontents

\section{Introduction}

A new class of quantum invariants of links and three-manifolds was introduced in \cite{akutsu1992invariants, murakami2008colored, geer2009modified, costantino2015quantum}, based on representation categories of quantum groups that may be non-semisimple and/or have vanishing quantum dimensions. These invariants generalize ``classic'' quantum invariants of knots and three-manifolds, such as the colored Jones and WRT invariants \cite{Turaev, Witten-Jones, RT}, which are instead constructed from semisimple representation categories where quantum dimensions are all nonzero. This paper arose from studying various properties of the new class of invariants, theoretically and via computations, with the goal of comparing their behavior to that of classic invariants.

A large part of our motivation came from physics. The colored Jones polynomials, HOMFLY polynomials, WRT invariants, etc. all have a physical origin in Chern-Simons theory with compact gauge group \cite{Witten-Jones}, which has led to many deep and unexpected insights over the past three decades.  An analogous physical origin for the new class of invariants --- a 3d continuum quantum field theory whose partition functions compute the new invariants --- has yet to be established. By investigating properties of the new invariants, one might hope to gain clues in identifying the missing 3d QFT's.

The property that we focus on in this paper concerns recursion relations. It was shown by Garoufalidis and L\^e in \cite{GL-J} that the sequence of colored Jones polynomials $\big(J_N^K(q)\big)_{N\in \N}$ of a knot $K$ always obey a finite-order recursion relation. More precisely, the function $J^K:\N\to \C[q,q^{-1}]$ generates a \emph{$q$-holonomic module} for the q-Weyl algebra
\be \mathbb E_1 = \C(q)[x^\pm,y^\pm]/(yx-qxy)\,, \ee
where $x$ and $y$ act on functions $f:\N\to\C(q)$ as multiplication by $q^N$ and shifting $N\mapsto N+1$, respectively. The theory of $q$-holonomic modules, central to the work of \cite{GL-J}, was developed by Sabbah \cite{Sabbah} and generalized classic work on D-modules by Bernstein, Sato, Kashiwara, and others. 

It was also conjectured in \cite{Gar-AJ} (and since confirmed in many examples \emph{e.g.} \cite{GaroufalidisKoutschan, GS-twist}) that the $q\to 1$ limit of any element $A(x,y;q)\in \E_1$ that annihilates the colored Jones function $J^K$ is divisible by the classical A-polynomial of $K$. Since the A-polynomial is defined using the classical $SL(2,\C)$ representation variety of the knot complement $S^3\backslash K$ \cite{Apoly}, this ``AJ conjecture'' established a new connection between colored Jones invariants and classical geometry. It remains an open conjecture.

The fact that the colored Jones polynomials should be annihilated by a recursion operator related to the A-polynomial was independently predicted by Gukov \cite{Gukov}, based on the physics of Chern-Simons theory. The approach of \cite{Gukov} was to analytically continue Chern-Simons theory with compact gauge group $SU(2)$ to a complex group $SL(2,\C)$; then an operator $A(x,y;q)$ providing recursion relations for the colored Jones was identified with an effective Hamiltonian that must annihilate the analytically continued Chern-Simons wavefunction. This operator had to be a quantization of the classical A-polynomial, which was the classical Hamiltonian of the system. (This insight was subsequently used in \cite{Gukov} to generalize the Volume Conjecture of \cite{Kashaev}.)

From a physical perspective, the presence of an operator $A(x,y;q)$ that quantizes the classical A-polynomial and annihilates quantum wavefunctions is now known to be an extremely robust feature of Chern-Simons theory with gauge group $SU(2)$ and many other versions of Chern-Simons theory with gauge group $SL(2,\C)$, including its analytic continuation (\emph{cf.} \cite{DGLZ, Dimofte-k, GukovManolescu}). In searching for a physical home for the new class of quantum invariants of \cite{akutsu1992invariants, geer2009modified, costantino2015quantum} it is therefore natural to ask whether they too satisfy recursion relations related to A-polynomials.

The invariants considered in this paper are defined using the representation category of the \emph{unrolled quantum group} $\UqgH_\rtu(\mathfrak{sl}_2)$ at the $2r$-th root of unity $\zeta_{2r}:=e^{\frac{i\pi}{r}}$, $r\in \N_{\geq 2}$. (See Section~\ref{sec:background} for details.) This quantum group admits a continuous family of `typical' representations $\{V_\alpha\}_{\alpha\in (\C\backslash \Z)\cup(-1+r\Z)}$ that are irreducible \emph{but} have vanishing quantum dimensions.

 Let $L$ be a framed, oriented link in $S^3$, with $n$ components colored by typical representations $V_{\alpha_1},...,V_{\alpha_n}$. It was shown in \cite{akutsu1992invariants,geer2009modified} how to overcome the problem of vanishing quantum dimensions to define a non-vanishing link invariant $N_L^r(\alpha_1,...,\alpha_n)$. 
 After restricting to $\alpha_i\in \C\backslash \Z$, it is useful to view these invariants as a family of functions
\be N_L^r : (\C\backslash\Z)^n \to \C\,,\qquad (r\in \N_{\geq 2})\,. \ee
They are in fact holomorphic and admit meromorphic continuations to $(\C/2r\Z)^n$.
Though this particular family of invariants can be defined using the systematic methods of \cite{geer2009modified}, they actually appeared much earlier in work of Akutsu, Deguchi, and Ohtsuki \cite{akutsu1992invariants}. Thus we call $N_L^r(\alpha)$ the ADO invariants.

We prove that the ADO invariants $N_L^r$ of any framed, oriented link $L$ are indeed $q$-holonomic. Moreover, in the case of a knot $L=K$, we prove that all recursion relations satisfied by the ADO invariants are also satisfied by the colored Jones function $J^K$. \medskip

While this paper was in final preparation, a physical interpretation of the ADO invariant appeared in work  \cite{Gukov:2020lqm} of Gukov, Hsin, Nakajima, Park, Pei, \& Sopenko. Therein, ADO invariants are related physically to a number of other invariants, including the recent homological blocks of \cite{GPV, GPPV, GukovManolescu}. It is conjectured in \cite[Sec 4]{Gukov:2020lqm} that ADO invariants obey the same recursion operations as Jones polynomials. Our results prove that this is indeed the case.

\subsection{Roots of unity, $q$-holonomic families, and Hamiltonian reduction}

It is not obvious what should be meant when considering whether the ADO invariants are $q$-holonomic. At each fixed $r$, the ADO invariant $N_L^r$ of an $n$-component link $L$ turns out to be quasi-periodic in each variable $\alpha_i$, with period $2r$. 
(We review this property in Proposition \ref{prop:struc} and Corollary \ref{cor:struc}.) The upshot is that the ADO invariant $N_L^r(\alpha_1,...,\alpha_n)$ at fixed $r$ will \emph{trivially} satisfy $n$ independent recursion relations, of the form
\be \label{triv-intro} \Big(\prod_{j=1}^n x_j^{-2rC_{ij}}y_i^{2r}-1\Big)N_L^r(\alpha)=0\,,\qquad i=1,...,n \ee
where each $x_i$ acts as multiplication by $\zeta_{2r}^{\alpha_i}:= e^{\frac{i\pi}r\alpha_i}$ and each $y_i$ acts as a shift $\alpha_i\mapsto \alpha_i+1$, and $C_{ij}$ is the integer linking matrix of $L$. These recursion relations, which depend only on the linking matrix, do not have a deep connection with the A-polynomial. 
 
To obtain interesting recursion relations, we work independently of the choice of $r$. This leads us to introduce the notion of a $q$-holonomic family. Let
\be \label{En-intro} \E_n = \C(q)[x_1^\pm,y_1^\pm,...,x_n^\pm,y_n^\pm]/(y_ix_j-q^{\delta_{ij}}x_jy_i) \ee
be a $q$-Weyl algebra in $n$ pairs of variables. Given an $n$-component link $L$ with ADO invariants $\{N_L^r(\alpha)\}_{r\geq 2}$, define an analog of the annihilation ideal $\mathcal I[N_L]\subseteq \E_n$ by
\be  \label{INL-intro} \mathcal I[N_L] = \{ A(x,y;q)\in \E_n\,|\, A(x,y;\zeta_{2r})N_L^r(\alpha) = 0\,\text{for all but finitely many $r\geq 2$}\}\,, \ee
with the usual action $x_i N_L^r(\alpha) = \zeta_{2r}^{\alpha_i}N_L^r(\alpha)$ and $y_i N_L^r(\alpha) = N_L^r(\alpha_1,...,\alpha_i+1,...,\alpha_n)$.
(Note that the specialization of elements of $\E_n$ to $q=\zeta_{2r}$ may not be defined at some finite number of $r$'s, which we ignore on the RHS of \eqref{INL-intro}.) We prove

\bigskip
\noindent\textbf{Theorem \ref{thm:ADO}} \emph{For any framed, oriented link $L$, the left $\E_n$-module $\E_n/\mathcal I[N_L]$ is $q$-holonomic.}
\bigskip

In particular, this implies that each ADO function $N_L^r(\alpha)$ satisfies $n$ independent recursion relations, which come from operators $A(x,y;q)\in \E_n$ that are \emph{independent} of $r$.

Our method of proof is to first show that the ADO invariants $N_L^r(\alpha)$ may be lifted (or analytically continued) to functions $G_{\mathbb D}(r;x_1,...,x_n,z_{11},z_{12},...,z_{nn};q)$ of $1+n+\frac12 n(n+1)$ variables $r,x_i,z_{ij}=z_{ji}$, as well as $q$, in such a way that
\be \label{spec-intro} N_L^r(\alpha) =  G_{\mathbb D}\big(r;\zeta_{2r}^{\alpha_1},...,\zeta_{2r}^{\alpha_n},\zeta_{2r}^{\alpha_1^2/2},\zeta_{2r}^{\alpha_1\alpha_2/2},...,\zeta_{2r}^{\alpha_n^2/2}; \zeta_{2r}\big)\,.  \ee
The lift from $N_L^r$ to $G_{\mathbb D}$ is not canonical, and $G_{\mathbb D}$ is \emph{not} a link invariant. It is defined in Section~\ref{sec:GD} using a choice of diagram $\mathbb D$ for a $(1,1)$-tangle whose closure is $L$. 

The virtue of $G_{\mathbb D}$ is that it is relatively straightforward to prove it generates a $q$-holonomic module for the $q$-Weyl algebra $\E_{n+1}$, in the same $n$ pairs of generators $x_i,y_i$ as \eqref{En-intro} together with a final pair $\hat x,\hat y$ that act as multiplication by $q^r$ and shift $r\mapsto r+1$. The proof that $G_{\mathbb D}$ is $q$-holonomic (contained in Section \ref{sec:preADO}) is a simple generalization of the original work of \cite{GL-J}. 

We then argue in Section \ref{sec:spec} that the specialization \eqref{spec-intro}, which in particular sets $q$ to be a $2r$-th root of unity, may be analyzed using a version of \emph{quantum Hamiltonian reduction}. The Hamiltonian reduction reduces $\E_{n+1}$ to $\E_n$ by eliminating the shift $\hat y$ in $r$ and setting $\hat x=\zeta_{2r}^r = -1$. It takes the annihilation ideal of $G_{\mathbb D}$ in $\E_{n+1}$ and explicitly constructs elements of our desired ideal $\mathcal I[N_L]$. We give a self-contained proof that the relevant Hamiltonian reduction preserves $q$-holonomic modules in Appendix \ref{app:red}.

Our result in Theorem \ref{thm:ADO} that the family of ADO invariants is $q$-holonomic would not be interesting if the elements $A(x,y;q)$ of $\mathcal I[N_L]$ were as trivial as the recursion relations in \eqref{triv-intro}. We prove in Section \ref{sec:Jones} that this is not the case, since the ideal $\mathcal I[N_L]$ is included in the annihilation ideal of the colored Jones function up to rescaling of variables.

More concretely, suppose that $L=K$ is an oriented knot with framing $f$, and $J_N^K(q)$ are its colored Jones polynomials, normalized so that $J_N^{unknot}(q) = (q^N-q^{-N})/(q-q^{-1})$. 

\bigskip
\noindent\textbf{Theorem \ref{thm:Jones}} \emph{For every element $A(x,y;q)\in \mathcal I[N_K]$, we have $A(q^{-1}x,(-1)^{f+1}y;q)J_N^K = 0$. 
}
\bigskip

\noindent This result follows fairly quickly from a relation between ADO invariants and colored Jones polynomials discussed in \cite{CGP-relations}. The relation is representation-theoretic in origin: as the parameter $\alpha$ of a typical module $V_\alpha$ for the unrolled quantum group approaches an integer $N-1\in \Z\backslash r\Z$, the module becomes reducible and its simple quotient coincides with the module used in defining the Jones polynomial.

Recently it has been shown in \cite{Willetts, Beliakova} that \emph{both} ADO and colored Jones invariants of links may be obtain by specializations of more universal invariants valued in the Habiro ring \cite{Habiro-cyc, Habiro-sl2}. One might expect that such relations lead to an independent proof that the family of ADO invariants is $q$-holonomic, with recursion relations equivalent to those satisfied by the colored Jones. Indeed, Sonny Willetts proves in Theorem 66 of the upcoming revised version of \cite{Willetts} that every element in the annihilation ideal of the colored Jones of a knot will also annihilate the family of ADO invariants. This is a converse to our Theorem \ref{thm:Jones}. Taken together, the two results establish that the annihilation ideals of the colored Jones and ADO invariants are equivalent.

\subsection{Example: figure-eight knot}

The Jones polynomials $\big(J^{\mathbf{4_1}}_N(q)\big)_{N\in \N}$ of the zero-framed figure-eight knot,
\be \begin{array}{l} J^{\mathbf{4_1}}_1(q) = 1\,,\qquad J^{\mathbf{4_1}}_2(q) = q^5+q^{-5}\,, \\[.1cm]
 J^{\mathbf{4_1}}_3(q)    = q^{14}-q^{10}+q^2+1+q^{-2}-q^{-10}+q^{-14} \\[.1cm]
 J^{\mathbf{4_1}}_4(q)  = q^{27}-q^{23}-q^{21}+q^{17}+q^{11}+q^9+q^{-9}+q^{-11}+q^{-17}-q^{-21}-q^{-23}+q^{27}\,, \\
 \ldots\,
 \end{array}
\ee
normalized so that $J^{unknot}_N(q)= \frac{q^N-q^{-N}}{q-q^{-1}}$, 
satisfy the 2nd-order inhomogeneous recursion
\be \label{41-inhom} (q-q^{-1})A(x,y;q)  J^{\mathbf{4_1}}_N(q) =B(q^N;q)\,,  \ee
where%
\footnote{This differs slightly from the recursion relation found in \cite{GL-J}, only because of the normalization of the colored Jones polynomials we are using here. The recursions are completely equivalent.}
\be \label{41-AB} \begin{array}{l} A(x,y;q) = (\tfrac{x^2}{q}-\tfrac{q}{x^2})y-(x^2-\tfrac{1}{x^2})(x^4-x^2-(q^2+q^{-2})-x^{-2}+x^{-4})+(qx^2-\tfrac{q}{qx^2})y^{-1} \\[.2cm]
 B(x;q) = (x+\tfrac1 x)(qx^2-\tfrac{1}{qx^2})(\tfrac{x^2}{q}-\tfrac{q}{x^2})\,, \end{array} \ee
and $x$ and $y$ act as multiplication by $q^N$ and shift $N\mapsto N+1$, respectively. The inhomogeneous recursion above implies the existence of a homogeneous recursion of one order higher,
\be \label{41-hom}  \big[B(x;q)y - B(qx;q)\big]A(x,y;q) J_N^{\mathbf{4_1}}(q) = 0\,. \ee
The operator $\widetilde A(x,y;q):= \big[B(x;q)y - B(qx;q)\big]A(x,y;q)$ generates the annihilation ideal of the colored Jones. At $q=1$, it is easy to see that $\widetilde A(m,\ell;q=1) = (m+m^{-1})(m^2-m^{-2})^3(\ell-1) \big(\ell-(m^4-m^2-2-m^{-2}+m^{-4})+\ell^{-1}\big)$ is divisible by the A-polynomial of the figure-eight knot, namely $(\ell-1) (\ell-(m^4-m^2-2-m^{-2}+m^{-4})+\ell^{-1})$.

A compact formula for the ADO invariants of the zero-framed figure-eight knot was given in~\cite{Murakami-cone}; adjusted for our conventions in this paper, it reads
\be
N_{\mathbf{4_1}}^r(\alpha-1) = \frac{-i^{1-r}}{x^r-x^{-r}} \sum_{k=0}^{r-1}x^{2k+1}(q^{-2k}x^{-2};q^2)_{2k+1}\Big|\raisebox{-.3cm}{$x=\zeta_{2r}^\alpha,q=\zeta_{2r}$}
\ee
Letting $\hat N_{\mathbf{4_1}}^r(\alpha):= i^{1-r}(x^r-x^{-r})N_{\mathbf{4_1}}^r(\alpha-1)$, the first few ADO invariants are
\be \begin{array}{l@{\quad}l}
\hat N_{\mathbf{4_1}}^2(\alpha) = (x+x^{-1})(x^2+3+x^{-2})& (x=e^{\frac{i\pi}{2}\alpha}) \\[.1cm]
\hat N_{\mathbf{4_1}}^3(\alpha) = (x+x^{-1})(x^4+3x^2+5-3x^{-2}+x^{-4}) &  (x=e^{\frac{i\pi}{3}\alpha}) \\[.1cm]
\hat N_{\mathbf{4_1}}^4(\alpha) = (x-x^{-1})(x^2+1+x^{-1})^3  &  (x=e^{\frac{i\pi}{4}\alpha}) \\[.1cm]
\hat N_{\mathbf{4_1}}^5(\alpha) = (x-x^{-1})(x+x^{-1})^2\big[x^6+x^4 & \\[.1cm]
\hspace{.5in} +(3+q^2-q^3)(x^2+x^{-2})+(2-q^2+q^3)+x^{-4}+x^{-6}\big] &  (x=e^{\frac{i\pi}{5}\alpha},\, q=e^{\frac{i\pi}{5}})
\end{array}
\ee
Further values appear in Appendix \ref{app:comp}. We verify for each $2\leq r\leq 20$ that
\be A(x,y;\zeta_{2r})\hat N_{\mathbf{4_1}}^r(\alpha) = -(\zeta_{2r}^{2r\alpha}-3+\zeta_{2r}^{-2r\alpha})B(\zeta_{2r}^\alpha,\zeta_{2r})\,, \ee
for exactly the \emph{same} $A$ and $B$ as in \eqref{41-AB}, with $x$ and $y$ now acting as multiplication by $\zeta_{2r}^\alpha$ and shift $\alpha\mapsto\alpha+1$, respectively. These inhomogeneous recursions imply that for each $r$ the ADO invariant satisfies a homogeneous recursion
\be \label{41-AN} \widetilde A(x,y;\zeta_{2r}) \hat N_{\mathbf{4_1}}^r(\alpha) = 0  \qquad r\in \N_{\geq 2}\ee
for exactly the same $\widetilde A (x,y;q)= \big[B(x;q)y - B(qx;q)\big]A(x,y;q)$ that annihilated the colored Jones. Note that \eqref{41-AN} is equivalent to $\widetilde A(qx,-y;\zeta_{2r}) N_{\mathbf{4_1}}^r(\alpha) = 0$ in the `un-hatted' normalization, in agreement with Theorem \ref{thm:Jones}.

Further examples of inhomogeneous and homogeneous recursions for the trefoil and $\mathbf{5_2}$ knots are collected in Appendix \ref{app:comp}.

\subsection{Acknowledgements}

We would like to thank Sam Gunningham for many discussions and advice on formal aspects of $q$-holonomic modules.
We would also like to thank Sonny Willetts for sharing his unpublished results on recursions for the ADO invariant.
This work is supported by the NSF FRG Collaborative Research Grant DMS-1664387. Research of N. Geer was also partially supported by NSF grant  DMS-1452093.

\section{Background}
\label{sec:background}

\subsection{An extension of the Drinfel'd-Jimbo algebra}

	Here we consider the aforementioned unrolled quantum group.  This object was first fully established in \cite{geer2009modified}, though ideas of its formulation were already present in \cite{akutsu1992invariants,ohtsuki2002quantum}.  For more details about the unrolled quantum group and its representation theory see \cite{costantino2015some,geer2018trace}.

	Let $q$ be a formal variable.   Fix a positive integer $r\geq 2$, and let $\rtu=\exp\big(\frac{\pi\sqrt{-1}}{r}\big)$. Throughout this paper we use the notation $\rtu^x:=  e^{\frac{\pi\sqrt{-1}}{r}x}$ for any $x\in \C$.
	Let $\Fq$ be the subring of $\C(q)$ consisting of elements with no
	poles at $q=\rtu$.  A $\Fq$-module
	can be specialized at $q=\rtu$ by tensoring with the module $\Fq/(q-\rtu)$.
	
	Consider the $\Fq$-algebra $\Uqg_q=\Uqg_q(\mathfrak{sl}_2)$  generated by $E, F, K, K^{-1}$ with relations
	\begin{equation}\label{eq:Drinfled-Jimbo}
	KF = q^{-2}FK,\quad KE = q^{2}EK,\quad KK^{-1} = K^{-1}K = 1, \text{ and}\quad [E,F] = \dfrac{K-K^{-1}}{q-q^{-1}}.
	\end{equation}
	This is a Hopf algebra with co-product, co-unit, and antipode defined on generators by:
	\begin{equation}
	\begin{aligned}
	\triangle(E) &= 1\otimes E + E \otimes K, \qquad& \varepsilon(E) &= 0, &\qquad  S(E) &= -EK^{-1}, \\
	\triangle(F) &= K^{-1}\otimes F + F \otimes 1, & \varepsilon(F) &= 0, & S(F) &= -KF, \\
	\triangle(K) &= K\otimes K, & \varepsilon(K) &=1, & S(K) &= K^{-1}.
	\end{aligned}
	\end{equation}
	The Hopf algebra $\Uqg_q$ is usually called the Drinfeld-Jimbo quantum group.  
	
	The \emph{unrolled quantum group} $\UqgH_\rtu=\UqgH_\rtu(\mathfrak{sl}_2)$ is the $\C$-algebra generated by
	$E, F, K, K^{-1}, H$ with Relations
	\eqref{eq:Drinfled-Jimbo} specialized to $q=\rtu$, together with the relations
	\begin{equation}
	HK= KH, \qquad [H,E] = 2E,\qquad [H,F] = -2F.
	\end{equation}
	The algebra $\UqgH_\rtu$ is a Hopf algebra with coproduct, counit and antipode defined as above on $K^\pm,E,F$ and defined on the element $H$ as 
	\begin{equation}
	\triangle(H) = H\otimes 1 + 1\otimes H,\qquad \varepsilon(H) = 0, \qquad S(H) = -H.
	\end{equation}
To connect with the ADO invariant, we will further pass to the central quotient
\be  \Ubar=\Ubar(\mathfrak{sl}_2) := \UqgH_\rtu/(E^r,F^r)\,. \ee

	\subsubsection{Representations of \texorpdfstring{$\Ubar$}{U}}
	
	Let $V$ be a finite-dimensional $\Ubar$ module. An eigenvalue $\lambda \in \mathbb{C}$ of $H$ is called a \emph{weight} and the associated eigenspace is called the \emph{weight space}. We say $V$ is a \emph{weight module} if it splits as a direct sum of weight spaces and $q^H = K$ as operators on $V$, i.e. $Kv = \rtu^\lambda v$ for any weight vector $v$ with $Hv = \lambda v$.  Let $\cat$ denote the category of finite dimensional weight modules of $\Ubar$. 
	
	Consider the following two families of modules.  For $\alpha\in \C$, let $V_\alpha$ be the object in $\cat$  with a basis $\{v_0,\ldots,v_{r-1}\}$ on which the $\Ubar$-action is given by
	\begin{align}
	Ev_i &=
	\dfrac{\rtu^{\alpha - i +1}-\rtu^{-(\alpha - i +1)}}{\rtu-\rtu^{-1}}v_{i-1}  , &%
	Fv_i &= 
	\dfrac{\rtu^{i+1}-\rtu^{-(i+1)}}{\rtu-\rtu^{-1}}v_{i+1}, \\[\medskipamount]
	Hv_i &= (\alpha - 2i) v_i, & Kv_i &= \rtu^{\alpha - 2i}v_i \notag
	\end{align}
	 where $v_{-1} = v_{r} = 0$. When $\alpha \in (\C\setminus \Z) \cup (-1+r\Z) $ the module $V_\alpha$ is simple and called \emph{typical}.  As we will now discuss, when $\alpha \in \Z \setminus (-1+r\Z) $ the module $V_\alpha$ is not decomposable --- it has a simple submodule which is not a direct summand.

	For each $n \in \Z_{\geq 0}$, let $S_n^{q}$ be the usual
	$(n+1)$-dimensional irreducible highest weight $\Uqg_q$-module with
	highest weight $n$. 
	The module $S_{n}^{q}$ has a basis $\{s_0, s_1,...,s_n\}$ on which the $\Uqg_q$-action is given by $ Ks_i = q^{n - 2i}s_i$ and 
	\begin{align}\label{E:RelSnq}
	Es_i &=
	\dfrac{q^{n - i +1}-q^{-(n - i +1)}}{q-q^{-1}}s_{i-1}, &
	Fs_i &= 
	\dfrac{q^{i+1}-q^{-(i+1)}}{q-q^{-1}}s_{i+1} 
	\end{align}
	where $s_{-1}=s_{n+1}=0$.  
	If  $n \in \{0,\ldots,r-1\}$ then by setting $q=\rtu$ and $Hs_i = (n - 2i) s_i$, the $\Uqg_q$-module $S_{n}^{q}$ becomes a simple $\Ubar$-module $S_{n}$.  In general, if $m \in \{0,\ldots,r-1\}$ and $k\in \Z$ then we can define a simple $(m+1)$-dimensional $\Ubar$-module $S_{m+kr}$  with basis $\{s_0,...,s_{m}\}$ on which the $\Ubar$-action is given by
	\begin{align*}
	Hs_i &= (m+ kr- 2i) s_i,&  Ks_i &= q^{m+kr - 2i}s_i 
	\end{align*} and \eqref{E:RelSnq} with $q=\rtu$ and $n=m+rk$ (here we set $s_{-1}=s_{m+1}=0$).  
	Notice that the definitions of $V_{kr-1}$ and $S_{kr-1}$ coincide.
	
	\begin{lemma}
		Every simple module of $\cat$ is isomorphic to exactly one of the
		modules in the list:
		\begin{itemize}
			\item  $S_{n +kr}$, for $n=0,\cdots, r-2$ and $k\in \Z$,  
			\item  $V_\alpha$  for  $\alpha\in(\C\setminus \Z)\cup
			(-1+r\Z)$. 
		\end{itemize}
	\end{lemma}
	\begin{proof}
		An argument analogous to that of finite dimensional  $\mathfrak{sl}_2$-modules (see for example \cite[Section V.4]{kassel1995quantum}) shows the following: (1) every non-zero  $\Ubar$-module in $\cat$ has a highest weight vector and (2) if $W$ is a simple $\Ubar$-module in $\cat$ then it is uniquely
		determined, up to isomorphism, by its highest weight $\lambda \in
		\C$.  The lemma then follows from the fact that the highest weights of
		modules in the above list are in bijection with the elements of $\C$.
	\end{proof} 
	
	When $\alpha = n+kr$, $n=0,...,r-2$, the module $V_\alpha$ is no longer irreducible.  Instead, there is a non-split short exact sequence 
	$$
	0\to S_{n +kr-2(n+1)}\to V_{n+kr} \to S_{n +kr}\to 0
	$$ 
	where the first morphism is determined by sending the highest weight vector of $S_{n +kr-2(n+1)}^{\rtu}$ to $v_{n+1}$ and the second morphism is given by sending the highest weight vector $V_{n+kr} $ to the highest weight vector of $S_{n +kr}^{\rtu}$.  
	The families 
	$$\{V_\alpha\}_{\alpha \in (\C\setminus \Z) \cup (-1+r\Z)} \text{ and } \{S_n^q\}_{n \in \Z_{\geq 0}}$$ 
	are used to define the ADO invariant and colored Jones polynomial, respectively.

	\subsubsection{The ribbon structure on \texorpdfstring{$\cat$}{C}}\label{SS:RibbonStructure}
	
	Here we recall that $\cat$ is a ribbon category, for details see for example \cite{geer2009modified,costantino2015some,geer2018trace}.  We will describe the ribbon structure in terms of left/right dualities and a braiding.  This formulation follows \cite{geer2018trace}, where it is shown that a ribbon category can be defined as a pivotal braided category satisfying certain compatibility constraints on the natural twist morphism defined from the braiding and dualities.  This structure will be used later while defining link invariants.    
	
	Since $\Ubar$ is a Hopf algebra, $\cat$ is a monoidal category where
	the unit $\unit$ is the 1-dimensional trivial module $\C$.  Moreover,
	$\cat$ is $\C$-linear: hom-sets are $\C$-modules, the composition and
	tensor product of morphisms are $\C$-bilinear, and
	$\End_\cat(\unit)=\C\Id_\unit$.  We will often denote the unit
	$\unit$ by $\C$. 
	
	Let $V$ and $W$ be
	objects of $\cat$.  Let $\{v_i\}$ be a basis of $V$ and $\{v_i^*\}$ be
	the dual basis of $V^*=\Hom_\C(V,\C)$.  The duality morphisms of $\cat$ are
	\begin{align*}
	\coev_V :& \C \rightarrow V\otimes V^{*}, \text{ given by } 1 \mapsto \sum
	v_i\otimes v_i^*,  &
	\ev_V: & V^*\otimes V\rightarrow \C, \text{ given by }
	f\otimes w \mapsto f(w),\\
	\tcoev_V :& \C \rightarrow V^*\otimes V, \text{ given by } 1 \mapsto \sum
	v_i^*\otimes K^{r-1}v_i,  &
	\tev_V: & V\otimes V^*\rightarrow \C, \text{ given by }
	w\otimes f \mapsto f(K^{1-r}w).
	\end{align*}
	  As shown in \cite{geer2018trace} these morphisms define a pivotal structure on $\cat$.  
	Taking $V=V_\alpha$, the cup and cap morphisms can be written
\begin{align}\label{E:genericCupCap}
		\tcoev_{V_\alpha} :\,& \C \rightarrow V_\alpha^*\otimes V_\alpha, \text{ given by } 1 \mapsto \sum
	\rtu^{(r-1)(\alpha - 2i)} v_i^*\otimes  v_i, \notag \\
	\tev_V:\, & V\otimes V^*\rightarrow \C, \text{ given by }
	v_i\otimes v_j^* \mapsto \rtu^{(1-r)(\alpha - 2i)}\delta_{ij}
	\end{align}
	where $\delta_{ij}$ denotes the Kronecker delta.

	In \cite{ohtsuki2002quantum}, Ohtsuki truncates the usual formula of the $h$-adic
	quantum $R$-matrix to define an operator on $V\otimes W$ by
	\begin{equation}\label{eq:R-matrix}
	R = \rtu^{H\otimes H/2}\sum_{k=0}^{r-1} \frac{(\rtu-\rtu^{-1})^{2k}}{\rtu^k(\rtu^{-2};\rtu^{-2})_k} E^k \otimes F^k.
	\end{equation}
	where 
	\begin{equation}\label{E:PochhammerSym}
	(x;p)_n := \begin{cases}
	\displaystyle \prod_{k=0}^{n-1}(1-xp^k) & \text{if } n >0 \\
	0 & \text{otherwise} \\
	\end{cases}
	\end{equation}
	denotes the $q$-factorial (\emph{a.k.a.} $q$-Pochhammer symbol or quantum dilogarithm \cite{FaddeevKashaev}) and 
	$\rtu^{H\otimes H/2}$ is the operator given by  
	$$\rtu^{H\otimes H/2}(v\otimes v') =\rtu^{\lambda \lambda'/2}v\otimes v'$$
	for weight vectors $v$ and $v'$ of weights of $\lambda$ and
	$\lambda'$. We call $R$ the \emph{truncated $R$-matrix}.  It is not an element in $\Ubar\otimes \Ubar$; 
	however, the action of $R$ on the tensor product of two objects of 
	$\cat$ is a well-defined linear map.  Moreover, $R$ gives rise to a braiding $c_{V,W}:V\otimes W
	\rightarrow W \otimes V$ on $\cat$ defined by $v\otimes w \mapsto
	\tau(R(v\otimes w))$ where $\tau$ is the permutation $x\otimes
	y\mapsto y\otimes x$.  
	The inverse of the operator $R$ is
	\begin{equation}
  \label{eq:R}
  R^{-1}
 = \left(\sum_{k=0}^{r-1}(-1)^k  \frac{(\rtu-\rtu^{-1})^{2k}}{\rtu^{k^2}(\rtu^{-2};\rtu^{-2})_k} 
  E^k\otimes F^k \right)\rtu^{-H\otimes H/2}.
  \end{equation}

	For later reference, we compute the coefficients of the truncated $R$-matrix acting on $v_a \otimes w_b \in V_\alpha \otimes V_\beta$:
	\begin{align}\label{E:RmatrixAction}
	R(v_a&\otimes w_b) = q^{\frac{1}{2} H\otimes H}\sum_{k=0}^{r-1} \frac{(q-q^{-1})^{2k}}{q^k(q^{-2};q^{-2})_k} E^kv_a \otimes F^kw_b \notag \\
	&=  q^{\frac{1}{2}H\otimes H}\sum_{k=0}^{r-1} (-1)^kq^{k(\alpha -a-b -1)}\frac{(q^{-2(\alpha-a+1)};q^{-2})_k(q^{2(b+1)};q^2)_k}{(q^{-2};q^{-2})_k} v_{a-k}\otimes w_{b+k}  \notag \\
	&= \sum_{k=0}^{r-1} (-1)^kq^{k(\alpha -a-b-1 )}q^{\frac{1}{2}\lambda_{a-k}^\alpha\lambda_{b+k}^\beta} \frac{(q^{-2(\alpha-a+1)};q^{-2})_k(q^{2(b+1)};q^2)_k}{(q^{-2};q^{-2})_k}v_{a-k}\otimes w_{b+k}
	\end{align}
	where $\lambda_{b+k}^\beta = \beta-2(b+k)$ and $\lambda_{a-k}^\alpha = \alpha-2(a-k)$ are the weights of $w_{b+k}$ and $v_{a-k}$, respectively; and with $q=\rtu$.
	A similar calculation reveals the following coefficients for the inverse:
\begin{equation}\label{E:RInvMatrixAction}
R^{-1}(v_a\otimes w_b) =\sum_{k=0}^{r-1}(-1)^k q^{-\frac{1}{2}\lambda_a^\alpha\lambda_b^\beta} q^{k(\alpha -a -b +1)}\frac{(q^{-2(\alpha-a+1)};q^{-2})_k(q^{2(b+1)};q^2)_k}{(q^2;q^2)_k}v_{a-k}\otimes w_{b+k}\,,
\end{equation}
again with $q=\rtu$\,.

\subsection{The ADO invariant}\label{sec:ADO}

	A cousin of the Reshetikhin-Turaev family of invariants \cite{reshetikhin1990ribbon}, the ADO invariant is based on a functor from a category that formalizes link diagrams to a category of representations. The intent of this section is to provide a concise review of this invariant, along the way establishing notation. 
In this paper we always consider \emph{framed} and \emph{oriented} links and tangles.

	We consider framed oriented tangles whose components are colored (or labeled) by elements of~$\cat$.  Such tangles are called \emph{$\cat$-colored ribbons}.  Let $\Rib_\cat$ be the category of $\cat$-colored ribbons (for details see \cite[XIV.5.1]{kassel1995quantum}). The well-known Reshetikhin-Turaev construction defines a $\C$-linear functor 
	$$F:\Rib_\cat\to \cat$$
	(for details see \emph{e.g.} \cite{kassel1995quantum,turaev2016quantum}).  
	The value of any $\cat$-colored ribbon under $F$ can be computed using the six \emph{building blocks}, which are the morphisms $\overarrowsplus,\overarrowsminus,\downcurvearrowleft,\downcurvearrowright,\curvearrowleft,\curvearrowright$ in  $\Rib_\cat$.   The functor $F$ transforms these building blocks as follows:
	\begin{equation}\label{E:building-blocks}
	\begin{array}{l@{\quad}l@{\quad}l}
	F(\overarrowsplus) = \tau\circ R,& F(\downcurvearrowright) = \coev_V, & F(\downcurvearrowleft) = \tcoev_V\,, \\[.1cm]
	F(\overarrowsminus) = \tau\circ R^{-1},& F(\curvearrowright) = \ev_{V},& F(\curvearrowleft) = \tev_V.
	\end{array}
	\end{equation}
where $\tau(v\otimes w)=w\otimes v$ permutes the factors. Vertical lines are sent to the identity morphism and reversing the direction of an arrow is equivalent to coloring instead by the dual module. 
	
	If  $L$ is a link with some component labeled by a simple object $V\in \cat$ then by cutting this component we obtain a $(1,1)$ tangle $T$ whose two ends are labeled with $V$.   By definition $F(T)\in \End_\cat(V)$.  Since $V$ is simple, this endomorphism is the product of
	the identity $\Id_V:V\to V$ with an element $\langle T \rangle$ of the ground
	ring of $\cat$, i.e. $F(T)= \langle T \rangle \Id_V$.  In particular,
	\begin{equation}\label{E:DefF}
	\begin{aligned}
	F(L)&=F\left(\pic{0.6ex}{
		\begin{picture}(10,11)(1,0)
		\drawTr
		\put(2,4.5){$\mathsmall{T}$}
		\put(10,4){$\mathsmall{V}$}
		\end{picture}}\;\right)
	=\langle T \rangle\,F\left(\pic{0.6ex}{
		\begin{picture}(10,11)(1,0)
		\drawTr
		\put(0.3,4.5){$\mathsmall{\Id_V}$}
		\put(10,4){$\mathsmall{V}$}
		\end{picture}}\;\right)\\
	&=\langle T \rangle\, F\left( \pic{0.6ex}{
		\begin{picture}(8,7)(1,0)
		\drawQDim
		\put(7,2){$\mathsmall{V}$}
		\end{picture}}\right) 
	=\langle T \rangle (\tev_V\circ \coev_V) 
	= \langle T \rangle \qdim_\cat(V).
	\end{aligned}
	\end{equation}
	When $V=V_\alpha$ is typical then a direct calculation shows that quantum dimension vanishes:
	\be \qdim_\cat(V_\alpha) :=(\tev_{V_\alpha}\circ \coev_{V_\alpha}) =0\,. \ee
	See \cite{geer2009modified} for further details.   Thus, from Equation \eqref{E:DefF} we have that $F(L)=0$ if any component of $L$ is colored by a typical module $V_\alpha$.  
	
	In \cite{akutsu1992invariants}, Akutsu, Deguchi, and Ohtsuki showed that one can replace such a vanishing quantum dimension in Equation \eqref{E:DefF} with a non-zero scalar and obtain an invariant which is now known as the ADO invariant. This process was extended to a general theory in \cite{geer2009modified}.  We will briefly recall this construction.     
	
	Consider the function $\md$ from the set of typical modules to $\C$ given by
	\begin{equation}\label{E:modified-dimension}
	\md(V_\alpha) = \prod_{j=0}^{r-2}\dfrac{1}{\rtu^{\alpha + r - j} - \rtu^{-(\alpha + r -j)}} = -\rtu^{\frac12r(1-r)} \frac{\rtu^{\alpha+1}-\rtu^{-(\alpha+1)}}{\rtu^{r\alpha}-\rtu^{-r\alpha}}\,.
	\end{equation}
	This function is called the \emph{modified dimension}.  Let $L$ be a $\cat$-colored framed link with at least one component colored by a typical module $V_\alpha$.  Cutting this component as above, we obtain a $(1,1)$ tangle $T_{\alpha}$.  Then Proposition 35 of \cite{geer2009modified} implies that the assignment 
	$$
	L\mapsto F'(L):=\md(V_\alpha) \langle T_{\alpha}  \rangle
	$$
	is independent of the choice of the component to be cut and yields a well-defined isotopy invariant  of $L$.  This is the aforementioned ADO invariant.

In the remainder of this paper we will assume that \emph{every} strand of an $n$-strand link $L$ is colored by a typical module $V_{\alpha_i}$, $i=1,...,n$. Then it does not matter which strand is cut, and we can choose it without loss of generality to be the one labeled by $V_{\alpha_1}$. The corresponding ADO invariant defines a function
\be \label{def-ADO} N_L^r : (\C\backslash\Z)^n \to \C\,,\qquad N_L^r(\alpha_1,...,\alpha_n) = \md(V_{\alpha_1}) \langle T_{\alpha_1} \rangle\,. \ee
Establishing $q$-holonomic properties of this family of functions for $r\geq 2$ is the central focus of this paper.

The diagrammatic calculus summarized here will compute the ADO invariant in a blackboard framing. One may use the ribbon element in the category (or add extra loops to a diagram) to change to an arbitrary framing. Changing the framing of the $i$-th strand by $f$ units multiples the ADO invariant by a prefactor
\be \rtu^{\frac12 f[\alpha^2+2(1-r)\alpha]} \label{framing} \ee

\subsection{A two-step reconstruction of the ADO invariant}
\label{sec:GD}

For analyzing the $q$-holonomic properties of the ADO invariant $N_L^r$, it will be useful to split its construction into two steps:

\begin{itemize}
\item[1)] Cut an $n$-strand link $L$ to get a (1,1) tangle $T$ with a particular choice of diagram $\mathbb D$, arranged so that all crossings are of the form $\overarrowsplus$ or $\overarrowsminus$.  To this diagram we will associate a function
\be G_{\mathbb D} : \Z \to \mathbb V_n\,, \ee
where 
\be \label{defVn} \mathbb V_n :=\C(q^{\frac12},x_1^{\frac12},...,x_n^{\frac12},z_{11},z_{12},...,z_{nn}) \ee
is the field of rational functions in $1+n+\frac12 n(n+1) = \frac12(n+1)(n+2)$ formal variables $q^{\frac12}$, $\{x_i^{\frac12}\}_{i=1}^n$, and $\{z_{ij}\}_{i,j=1}^n$, with $z_{ij}=z_{ji}$. 
We call the function $G_{\mathbb D}$ the ``diagram invariant.''

\item[2)] For each $r\in \Z_{\geq 2}$ we specialize the variables in $G_{\mathbb D}(r)$ as
\be \label{ADO-spec}  q^{1/2} = \rtu^{1/2}\,,\quad x_i^{1/2}=\rtu^{\alpha_i/2}\,,\quad z_{ij}=\rtu^{\alpha_i\alpha_j/2}\,. \ee
to get the ADO invariant $N_L^r(\alpha)$. It will be clear from the construction of $G_{\mathbb D}$ that this specialization makes sense. More compactly: if we write $G_{\mathbb D}(r;x^{\frac12},z;q^{\frac12})$ to explicitly emphasize the dependence on $x,z,q$ for each value of $r$, then
\be \label{ADO-spec2} N_L^r(\alpha) = G_{\mathbb D}(r;\rtu^\alpha/2,\rtu^{\alpha\otimes \alpha/2};\rtu^{1/2})\,. \ee

\end{itemize}

To define  $G_{\mathbb D}$, suppose we are given a $(1,1)$ tangle diagram $\mathbb D$ arranged so that all crossings look like $\overarrowsplus$ or $\overarrowsminus$.

Let $n$ denote the number of components (strands) of the tangle. Let $m$ denote the number of arcs of the diagram $\mathbb D$, where by ``arc'' we mean a curve in the diagram that starts at one crossing and ends at the next, \emph{regardless} of whether the crossings go over or under. For example, the standard (1,1) tangle representing the trefoil has seven arcs (see Figure~\ref{fig:labeled-diagram}); and a general tangle diagram with $C$ crossings and $U$ disjoint flat unknot components (simple closed curves) has exactly $2C+U+1$ arcs.

\begin{figure}[htb]
\centering
\includegraphics[width=2in]{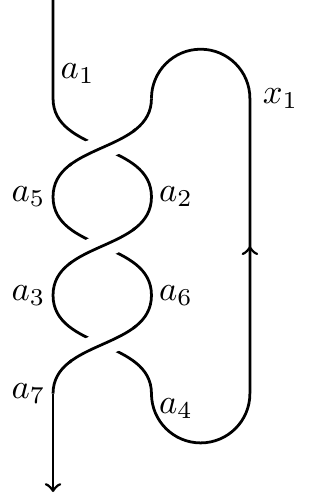}
\caption{Labeled tangle diagram whose closure is a trefoil knot.}
 \label{fig:labeled-diagram}
\end{figure}

We label each component of the tangle with a distinct variable $x_1,\ldots, x_n$. By convention the unique open component will be labeled by $x_1$.
We also label each arc with a distinct parameter $a_1,...,a_m$, which is to be thought of as an integer-valued variable.

We decompose the diagram $\mathbb D$ into crossings, cups, and caps. Then to each of these building blocks we associate a function of $m+1$ integral variables $a_1,...,a_m, r$, valued in $\mathbb V_n$, given by
\begin{subequations} \label{diag-piece-functions}
\begin{align}
\raisebox{-.2in}{\includegraphics[width=.5in]{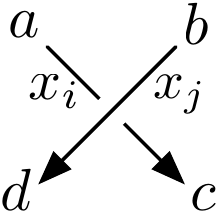}} &\rightsquigarrow 
	R_{c,d}^{a,b}[x_i,x_j] \label{R-tensor}   \\
	&:= \delta_{a-c,d-b} \vartheta_{c\leq a}\vartheta_{d\geq b}\,
	(-x_i)^{a-c}q^{(c-a)(a+b+1 )+2cd} z_{ij} x_i^{-d}x_j^{-c}\frac{(q^{2(a-1)}x_i^{-2};q^{-2})_{a-c}(q^{2(b+1)};q^2)_{a-c}}{(q^{-2};q^{-2})_{a-c}}  \notag
	\end{align}
	
\begin{align}
\raisebox{-.2in}{\includegraphics[width=.5in]{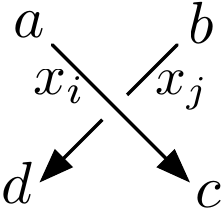}} &\rightsquigarrow 
	(R^{-1})_{c,d}^{a,b}[x_i,x_j] \label{Ri-tensor}  \\
	&:=\delta_{a-c,d-b}   \vartheta_{c\leq a}\vartheta_{d\geq b}\,
(-x_i)^k  q^{(c-a)(a +b -1)-2ab}z_{ij}^{-1}x_i^b x_j^a\frac{(q^{2(a-1)}x_i^{-2};q^{-2})_{a-c}(q^{2(b+1)};q^2)_{a-c}}{(q^{2};q^{2})_{a-c}}  \notag
	\end{align}
\be \label{e-tensor} \begin{array}{l@{\qquad\qquad}l}
 \raisebox{-.1in}{\includegraphics[width=.44in]{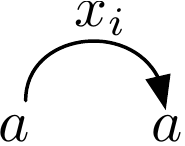}} \rightsquigarrow \epsilon_{a}[x_i]  = 1 
 & \raisebox{-.1in}{\includegraphics[width=.44in]{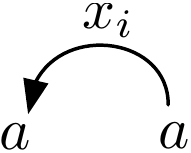}} \rightsquigarrow \epsilon^*_{a}[x_i] = q^{2a(r-1)}x_i^{1-r}   \\[.7cm]
 \raisebox{-.13in}{\includegraphics[width=.4in]{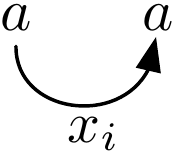}} \rightsquigarrow \eta_{a}[x_i] = 1
 & \raisebox{-.13in}{\includegraphics[width=.4in]{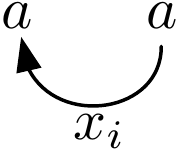}} \rightsquigarrow \eta^*_{a}[x_i] = q^{2a(1-r)} x_i^{r-1} 
\end{array} \ee 
Here we have used $a,b,c,d$ to denote the subset of arc variables $a_1,...,a_m$ present at a particular crossings. We have also used
\be \delta_{a,b} := \begin{cases} 1 & a=b \\ 0 & \text{otherwise}\end{cases}\,,\qquad \vartheta_{a\leq b} := \begin{cases} 1 & a\leq b \\ 0 &\text{otherwise}\end{cases}\,. \ee
We are thinking of each of the maps $R[x_i,x_j],R^{-1}[x_i,x_j],\epsilon[x_i],\epsilon^*[x_i],\eta[x_i],\eta^*[x_i]$ associated to particular crossings, cups, or caps as functions of the full set of arc variables $a_1,...,a_m$ together with $r$ --- though they are independent of the arc variables that do not appear in the building block under consideration. 

We similarly rewrite the modified quantum dimension \eqref{E:modified-dimension} associated to component $i$ as
\be \label{di} \md[x_i] =  \prod_{j=2}^{r}\dfrac{1}{q^j x_i- q^{-j} x_i^{-1}} = (-x)^{r-1} q^{\frac12 r(r+1)-1} \frac{1}{(q^4 x_i^2;q^2)_{r-1}}\,,
 \ee
\end{subequations}
thought of as a function of all $m+1$ integer variables, which depends non-trivially only on $r$.%
\footnote{One might wonder why we did not ``analytically continue'' the simpler formula on the RHS of \eqref{E:modified-dimension} to obtain $\md[x_i]=-q^{\frac12 r(1-r)} \frac{qx_i-(qx_i)^{-1}}{x_i^r-x_i^{-r}}$. The answer is that \eqref{di} turns out to be $q$-holonomic, whereas this latter expression is not!} Altogether, each function in \eqref{R-tensor}--\eqref{di} has domain $\Z^{m+1}$ and is valued in $\mathbb V_n$.

We define a function $G_{\mathbb D}^\times: \Z^{m+1} \to \mathbb V_n$ by multiplying together together the functions associated to every crossing, cup, and cap in the diagram; a function $\md[x_1]$ for the open link component (labeled $x_1$ by convention); and delta-functions $\delta_{a_1,0}$, $\delta_{a_m,0}$ for the two arcs at the open ends of the $(1,1)$ tangle (labeled, say, $a_1$ and $a_m$).
Schematically,
\be \label{GDx} G_{\mathbb D}^\times(a_1,...,a_m;r) = \md[x_1]\delta_{a_1^{},0}\delta_{a_m^{},0}\prod_{\overarrowsplus}R\prod_{\overarrowsminus}R^{-1}\prod_{\curvearrowright}\epsilon  \prod_{\curvearrowleft}\epsilon^* \prod_{\downcurvearrowright}\eta \prod_{\downcurvearrowleft}\eta^*\,. \ee
From this we define the \emph{diagram invariant} $G_{\mathbb D}:\Z\to \mathbb V_n$ as the multisum of $G_{\mathbb D}^\times$ over the arc variables
\be \label{GD} G_{\mathbb D}(r) := \sum_{a_1,...,a_m\in[0,r-1]^m} G_{\mathbb D}^\times(a_1,...,a_m;r) \ee

Once we fix $r\geq 2$ and specialize $q=\rtu$, $x_i=\rtu^{\alpha_i}$, and $z_{ij}=\rtu^{\alpha_i\alpha_j/2}$, each of the functions $R,R^{-1},\epsilon,\epsilon^*,\eta,\eta^*,\md$ above simply becomes a matrix element of the building blocks from \eqref{E:building-blocks}. This is easy to see by comparing with the formulas \eqref{E:RmatrixAction}, \eqref{E:RInvMatrixAction}, \eqref{E:genericCupCap}, \eqref{E:modified-dimension}. The multisum in \eqref{GD} reproduces the composition of building blocks \eqref{E:building-blocks}, summing over bases of the typical representations along the strands. Thus, altogether, the specialization of $G_{\mathbb D}(r)$ as in \eqref{ADO-spec} reproduces the ADO invariant $N_L^r(\alpha)$.

\medskip
\noindent\underline{\emph{Example}} \medskip

The labeled diagram of a (1,1) tangle whose closure is a trefoil knot is shown in Figure \ref{fig:labeled-diagram}. There are seven arcs with associated variables $a_1,\ldots, a_7$ and one component with associated variable $x_1$. The corresponding diagram invariant is
$$ G_\mathbb{D}(r) = \hspace{-1ex}  \sum_{a_1^{},\ldots,a_7^{} =0}^{r-1} \md[x_1] \delta_{a_1^{},0}\delta_{a_7^{},0} R^{a_1,a_4}_{a_2,a_5}[x_1,x_1]R^{a_5,a_2}_{a_6,a_3}[x_1,x_1]R^{a_3,a_6}_{a_4,a_7}[x_1,x_1] \epsilon^*_{a_4^{}}[x_1] \eta_{a_4^{}}[x_1]
$$

\subsection{Structural properties and relation to the colored Jones}
\label{sec:struc}

Several structural properties of the ADO invariant become manifest in the two-step construction of Section \ref{sec:GD}.

The only denominators that appear in the functions $R^\pm,\epsilon^{(*)},\eta^{(*)},\md$, which aren't just monomials in $x,z,q$, are $(q^4 x_i^2;q^2)_{r-1}$ in the modified dimension $\md[x_i]$ and $(q^{2};q^{2})_{a-c}$, $(q^{-2};q^{-2})_{a-c}$ in the R-matrices. A short exercise shows that the denominators in the R-matrices divide the numerators, as both $\frac{q^{2(b+1)};q^2)_k}{(q^2;q^2)_k}$ and $\frac{q^{2(b+1)};q^2)_{k}}{(q^{-2};q^{-2})_k}$ belong to $\C[q,q^{-1}]$ for all $k,b\in \Z_{\geq 0}$. Thus after simplification the only \emph{possible} denominator in $G_{\mathbb D}(r)$ is $x_1^r-x_1^{-r}$.
 
Moreover, only integral powers of $x,z,q$ appear; and the only place that $z_{ij}$ (resp. $z_{ij}^{-1}$) appears is as a prefactor in the $R$ function (resp. $R^{-1}$ function) for positive (resp. negative) crossing of components $i$ and $j$. Altogether this implies that

\begin{prop} \label{prop:struc} For each $r\geq 1$,
\be \label{G-struc} G_{\mathbb D}(r) \in \frac{1}{(q^4 x_1^2;q^2)_{r-1}} \prod_{i,j=1}^n z_{ij}^{C_{ij}}\; \C[x_1^{\pm},...,x_n^{\pm},q^{\pm}]\,, \ee
where $C_{ij}$ is the linking matrix of the original framed link $L$ ($C_{ii}$ being the framing of the $i$-th component).
\end{prop}

\begin{cor} \label{cor:struc} \hspace{1in}

\noindent (i) If $L$ is a knot ($n=1$ strands), the ADO invariant $N^r_L:\C\backslash\Z\to \C$ may be extended to a meromorphic function of $\alpha=\alpha_1\in \C$ with at most simple poles at each integer.

\noindent (ii) If $L$ is a link with $n>1$ strands, the ADO invariant $N^r_L:(\C\backslash\Z)^n\to \C$ may be extended to a holomorphic function of $\alpha=(\alpha_1,...,\alpha_n)\in \C^n$.

\noindent (iii) For any $n$, $N_L^r$ is quasi-periodic, satisfying
\be N_L^r(\alpha_1,...,\alpha_i+2r,...,\alpha_n) = \Big(\prod_{j=1}^n \rtu^{2rC_{ij}\alpha_j} \Big)N_L^r(\alpha_1,...,\alpha_i,...,\alpha_n) \label{quasi-per} \ee
In other words, $N_L^r$ is a section (holomorphic if $n>1$, meromorphic if $n=1$) of a complex line bundle on $(\C/2r\Z)^n$ determined by the linking matrix $C_{ij}$. 
\end{cor}

\noindent\emph{Proof}. For (i) we observe that after specializing $q=\rtu$ and $x_1=\rtu^\alpha$, the modified quantum dimension may be rewritten as on the RHS of \eqref{E:modified-dimension}, so the only denominator that could give rise to poles is $\rtu^{r\alpha}-\rtu^{-r\alpha}=2i\sin(\pi\alpha)$.

Part (ii) follows by recalling that the ADO invariant does not depend on the choice of strand that we cut to represent it as a (1,1) tangle. From \eqref{G-struc} and the same reasoning as in Part (i), it is clear that there are no poles in $\alpha_i$ for $i\neq 1$; and by constructing the ADO with a different choice of cut strand it follows that there can be no poles in $\alpha_1$ either.

Part (iii) follows from observing that, with the exception of the $z_{ij}^{C_{ij}}$ prefactors, the expression \eqref{G-struc} is a function of $x_i=\rtu^{\alpha_i}$, which are periodic. The prefactors $z_{ij}^{C_{ij}}=\rtu^{\frac12 C_{ij}\alpha_i\alpha_j}$ lead precisely to the quasi-periodicity \eqref{quasi-per}. \phantom{X}\hfill$\square$ \\

In the case of a knot, one might expect the residues of the poles at integer values of $\alpha$ to be related to colored Jones polynomials. This is because upon setting $\alpha=N-1$ (for $N\in \Z\backslash r\Z$) the typical module $V_\alpha$ becomes reducible and contains the module $S_{N-1}$ used to define the $N$-th colored Jones polynomial as a simple quotient. This expectation was made precise in Corollary 15 of \cite{CGP-relations}, which we restate here:

\begin{prop} \label{prop:Jones} (\cite[Cor. 15]{CGP-relations})
Let $K$ be an oriented knot with framing $f$. Let $r\geq 2$, and let $N\in N\backslash r \Z$. Let $J_N^K(q)\in \C[q^\pm]$ denote the $N$-th colored Jones polynomial of $K$, normalized so that $J_N^{unknot}(q) = (q^N-q^{-N})/(q-q^{-1})$. Then
\be \label{NRes} {\rm Res}_{\alpha=N-1} N_K^r(\alpha) = \frac{i^{1-r}}{\pi}\sin\Big(\frac{\pi}{r}\Big) (-1)^{N+f(N-1)} J_N(\rtu)\,. \ee
\end{prop}
 
\noindent Note that the prefactors on the RHS differ slightly from those in \cite{CGP-relations}. The $r$-dependent prefactors differ due to a different normalization for the modified dimension $\md[V_\alpha]$. The extra $(-1)^{N+f(N-1)}$ appear because the pivotal structure (and ribbon element) in the category $\Ubar(\mathfrak{sl}_2)$-mod discussed above --- the \emph{only} pivotal structure that exists for generic $\alpha\in \C$ --- differ from the pivotal structure (and ribbon element) used in the standard definitions of the colored Jones polynomial.

\medskip
\noindent\textbf{Remark.}
At $\alpha=N-1$ with $N\in r\Z$, the ADO invariant does not have a pole, and may simply be evaluated. In particular, it was shown some time ago by J. Murakami and H. Murakami \cite{MurakamiMurakami}  that $N_K^r(r-1)$  coincides with the renormalized Jones polynomial $\hat J_r(\zeta_{2r})$, where $\hat J_N(q) = \frac{q-q^{-1}}{q^N-q^{-N}}J_N(q)$, as well as with the Kashaev invariant \cite{Kashaev}. This observation allowed Kashaev's famous volume conjecture to be reformulated in terms of colored Jones polynomials.


\section{The diagram invariant is $q$-holonomic}
\label{sec:preADO}

Our next goal is to prove that the ``diagram invariant'' $G_{\mathbb D}$ defined in Section \ref{sec:GD} is $q$-holonomic. Specifically, for an $n$-component tangle, we show that $G_{\mathbb D}$ generates a $q$-holonomic module for a $q$-Weyl algebra with $n+1$ pairs of generators: $(x_i,y_i)_{i=1}^n$ acting as multiplication and $q$-shifts of the variables $x_i$ in $G_{\mathbb D}$, as well has $(\hat x,\hat y)$ acting as multiplication by $q^r$ and shift $r\mapsto r+1$.

Proving that $G_{\mathbb D}$ is $q$-holonomic in this sense is a straightforward generalization of the classic results of Garoufalidis and L\^{e} \cite{GL-J} on the Jones polynomial. We just adapt the methods there to the functional spaces to which $G_{\mathbb D}$ belongs. The main reason this is straightforward is that all the nice closure properties of $q$-holonomic functions under addition, multiplication, multisums, etc. are consequences of universal algebraic features of $q$-holonomic modules --- and in particular are completely independent of the actual functional spaces on which $q$-Weyl algebras are represented.

To maintain a reasonably self-contained and pedagogical exposition, we will review basic definitions and examples of $q$-holonomic modules in Section \ref{sec:qreview}, largely following the classic work of Sabbah \cite{Sabbah}, which in turn was based on work of Bernstein \cite{Bernstein}, Sato, Kashiwara, and others on D-modules.
(Other good references include the classic \cite{zeilberger1990holonomic, wilf1992algorithmic}, as well as the more recent survey \cite{GL-survey}. In modern days there are powerful derived methods available to study generalizations of $q$-Weyl modules and functors among them, such as \cite{KashiwaraSchapira-DQ}; but we will not require these methods.) In the process we will introduce the functional spaces $\mathcal V_{n,m}$ relevant for $G_{\mathbb D}$.

In Section \ref{sec:closure} we explain how standard closure properties of $q$-holonomic modules apply to $G_{\mathbb D}$. Then in Section \ref{sec:GDhol} we emulate \cite{GL-J} to prove that $G_{\mathbb D}$ is $q$-holonomic --- by verifying that all the building blocks of Section \ref{sec:GD} are $q$-holonomic and that their composition to form $G_{\mathbb D}$ preserves this property.

\subsection{$q$-holonomic modules and functions}
\label{sec:qreview}

\subsubsection{Basic definitions}

Let $k=\C(q)$ denote the field of fractions in a formal variable $q$. Recall the $q$-Weyl algebras in $n$ pairs of variables 
\be \begin{array}{r@{\,}c@{\,}l}
\W_n &=& k[x_1,...,x_n,y_1,...,y_n]/\text{rel}_q \\[.1cm]
\E_n &=&  k[x_1^{\pm 1},...,x_n^{\pm 1},y_1^{\pm1 },...,y_n^{\pm 1}]/\text{rel}_q^\pm
\end{array}
\ee
Namely, these consist of polynomials (resp. Laurent polynomials) in $2n$ non-commutative formal variables $x_i,y_i$ (resp. $x_i^\pm,y_i^\pm$), subject to the relations
\be \label{relq} \qquad  \text{rel}_q:\;\; \begin{array}{c}y_ix_j = q^{\delta_{ij}} x_j y_j \\[.1cm]
 x_ix_j = x_j x_i \\ [.1cm]
 y_iy_j = y_j y_i \end{array}
\qquad\quad  \text{rel}_q^\pm:\;\; \begin{array}{c}
y_i^{\varepsilon} x_j^{{\varepsilon'}} = q^{\varepsilon{\varepsilon'}\delta_{ij}} x_j^{{\varepsilon'}} y_j^{\varepsilon} \\ [.1cm]
x_i^\varepsilon x_j^{\varepsilon'} = x_j^{\varepsilon'} x_i^\varepsilon \\ [.1cm]
y_i^\varepsilon y_j^{\varepsilon'} = y_j^{\varepsilon'} y_i^\varepsilon
\end{array}
 \quad(\varepsilon,{\varepsilon'}\in \{\pm 1\})\ee
as well as the implicit relations $x_ix_i^{-1}=y_iy_i^{-1}=1$.

Both algebras have a notion of a $q$-holonomic module, though their respective definitions differ some.

The notion of $q$-holonomic $\mathbb{W}_n$ modules is based on homological dimension, which quantifies the quasi-periodicity of the module elements under the action of $\mathbb{W}_n$.  
Note that the algebra $\W_n$ has a non-negative ascending filtration $\CF_\bullet \W_n$ given by total degree in $x$ and $y$,
\be \CF_i \W_n = k\langle x^ay^b\;\text{s.t.}\; |a|+|b|\leq i \rangle\,, \ee
where we write $|a|=\sum_j a_j$ for a multi-index $a=(a_1,...,a_n)$. This is often referred to as the ``Bernstein Filtration." Given a left $\W_n$-module $M$, an ascending filtration $\CF_\bullet M$ is called a ``good filtration'' if the associated Rees module is a finitely generated module for the Rees algebra of $\W_n$; in particular, this implies that the filtrations on $\W_n$ and $M$ are compatible (\emph{i.e.} $\CF_i\W_n \cdot \CF_j M \subseteq \CF_{i+j} M$), and that each $\CF_i M$ is finite-dimensional. 

Remarkably, for every good filtration there exists a (necessarily unique) polynomial $p$, the ``Hilbert polynomial'', such that $\text{dim}_k \CF_i M = p(i)$ for $i \gg 0$. Moreover, the degree of this polynomial, denoted  $d(M)$ and called the \emph{homological dimension} of $M$, is independent of the choice of good filtration. In other words, $d(M)$ is the polynomial order of growth of the filtered components of any good filtration.

The $q$-analogue of Bernstein's inequality guarantees that if $M$ is finitely generated and has no monomial torsion%
\footnote{Given a left $\W_n$-module $M$, its \emph{monomial torsion} $\text{mtor}(M) \subseteq M$ is the subspace consisting of $v\in M$ such that $x^a y^b\,v=0$ for some monomial $x^ay^b := x_1^{a_1}\cdots x_n^{a_n}y_1^{a_1}\cdots y_n^{a_n}\in \W_n$.} %
then $d(M)\geq n$. We are interested in the case when the homological dimension is as small as possible.

\begin{definition} \label{def:W-hol} A left $\W_n$-module $M$ is called \emph{$q$-holonomic} if it is finitely generated, has no monomial torsion, and either $M=0$ or $d(M)=n$.
\end{definition}

Since elements of $\E_n$ may have arbitrarily large negative degree, the components of the Bernstein filtration will be infinite dimensional. Being $q$-holonomic is instead defined in terms of homological codimension. Given a left $\E_n$-module $M$, its \emph{homological codimension} $c(M)$ is the smallest integer $i$ such that $\text{Ext}_{\E_n}^i(M,\E_n)\neq0$.

\begin{definition} \label{def:E-hol}
A left $\E_n$-module $M$ is called $q$-holonomic if it is finitely generated and $M=0$ or $c(M)=n$.
\end{definition}

Results in \cite[Sec. 2]{Sabbah} show that for any finitely-generated $\E_n$-module $M$ one has $c(M)\leq n$ (an analogue of Bernstein's inequality), and that $M$ is in fact $q$-holonomic if and only if $\text{Ext}_{\E_n}^i(M,\E_n)=0$ for all $i\neq n$.

There is a close relationship between $\E_n$ and $\W_n$ modules. First, $\E_n$ has a natural right $\W_n$-module structure, which provides a map from (left) $\W_n$ modules to $\E_n$ modules, {\it i.e.}
\be \label{WE-induce}  \begin{array}{ccc}
   \W_n\text{-mod} & \to & \E_n\text{-mod} \\[.1cm]
   M &\mapsto & \E_n\otimes_{\W_n} M\,. \end{array}
\ee
Note that the kernel of this map consists precisely of $\W_n$-modules with monomial torsion. Conversely, any finitely-generated $\E_n$-module $M$ can be written as $M = \E_n \otimes_{\W_n} N$ for some $N$ (just take $N$ to be the $\W_n$-span of the generators of $M$).
A simple result of \cite[Sec. 2]{Sabbah} is

\begin{prop} \label{prop:EW} A left $\E_n$-module $M$ is $q$-holonomic if and only if there exists a $q$-holonomic left $\W_n$ module $N$ with $M = \E_n\otimes_{W_n}N$.
\end{prop}

\subsubsection{Cyclic modules}

We will mainly be interested in cyclic modules, \emph{i.e.} modules of the form $M = \E_n v$ or $N=\W_n v$ generated by a single element $v$. In the case of $\W_n$-modules, a useful observation is that every cyclic module has a canonical good filtration, given by
\be \CF_i N := (\CF_i  \W_n) v\,. \label{F-can} \ee
In the case of $\E_n$-modules, another structural result of \cite[Sec. 2]{Sabbah} shows that
\begin{prop} \label{prop:cyclic}
Every $q$-holonomic $\E_n$-module is cyclic.
\end{prop}

We recall that any cyclic module may be written in the form
\be M = \E_n \big/ \text{Ann}_{\E_n}(v) \qquad \text{or}\qquad  N = \W_n \big/ \text{Ann}_{\W_n}(v)\,, \ee
where the \emph{annihilator ideal} $\text{Ann}_A(v) =\{a\in A\;\text{s.t.}\; a v =0\}$ is the left ideal in the algebra $A=\E_n$ or $\W_n$ consisting of elements that kill the generator.

For a cyclic module $M$, being $q$-holonomic roughly implies that the annihilator ideal has at least $n$ independent generators. This can be made precise by introducing the characteristic variety $\text{char}M\in (\C^*)^{2n}$; by (\emph{e.g.}) Prop. 7.1.9 of \cite{KashiwaraSchapira-DQ}, $M$ is $q$-holonomic if and only if $\dim (\text{char}M) = n$. The corresponding statement for D-modules is a classic result in the theory, \emph{cf.} \cite{Kashiwara-Bfunctions}. A weaker, specialized result, which is sufficient for all the examples we need to consider in this paper, is the following:
\begin{lemma} \label{lemma:hol-y}
Let $M=\E_n v$ be a cyclic $\E_n$-module whose annihilator ideal contains elements of the form $p_j(x)y_j{}^{d_j}+q_j(x)$ for each $j=1,...,n$, with $p_j(x),q_j(x)\in \C(q)[x_1,...,x_n]$, $p_j,q_j\neq 0$. Then $M$ is $q$-holonomic.
\end{lemma}

\noindent\emph{Proof.}
We will prove that the associated $\W_n$-module $N=\W_n v$ is $q$-holonomic, by showing that the dimensions of the filtered components $\CF_iN = (\CF_i\W_n)v$ obey $\dim \CF_iN \leq C i^n$ for some fixed constant $C$. Then it follows from Prop. \ref{prop:EW} that $M=\E_nv$ is $q$-holonomic.

Choose any $i\geq \max\{n,d_1,...,d_n\}$. The filtered component $\CF_iN$ is certainly spanned by all the monomials $x^ay^bv:= x_1^{a_1}...x_n^{a_n}y_1^{b_1}...y_n^{b_n}v$ with $|a|+|b|\leq i$. However, the relations
\be (p_j(x)y_j{}^{d_j}+q_i(x))v=0\,,\qquad  j=1,...,n \label{relpq} \ee
make some of these monomials redundant, and reduce the dimension. Let $c_j = \deg_{x_j}p_j(x)$. Then, for any $j$, we observe that if $x^ay^bv$ is divisible by $y_j^{d_j}$, the relations \eqref{relpq} imply that it is sufficient to consider $x^a$ such that $\deg_{x_j}x^a < c_j$. In other words, $\CF_i N$ is spanned by
\be x^ay^b v\quad\text{s.t.}\quad a_j,b_j\geq 0\;\forall\,j\,,\quad  |a|+|b|\leq i\,,\quad\text{and}\quad \forall\,j,\; b_j\geq d_j\;\Rightarrow\; a_j< c_j\,. \label{mon-count} \ee

We seek an upper bound for dimension of this space of monomials. To simplify things, let $\bar d = \text{max}\{d_1,...,d_n\}$ and $\bar c = \max\{c_1,...,c_n\}$. For each $0\leq m \leq n$, let
\begin{equation*}
S_m := \left\{ (a,b) \left| \begin{aligned}
&a_j,b_j\leq i\;\;\forall\,j, \\&\text{$b_j\geq \bar d$ for exactly $m$ values of $j$,} \\
&\text{and}\;\; b_j \geq \bar d\;\Rightarrow\; a_j < \bar c
\end{aligned}\right.\right\}
\end{equation*}
Then $\bigcup_{m=0}^n S_m$ contains the set of $(a,b)$ satisfying \eqref{mon-count}, and it is straightforward to count
\begin{align} 
 |S_m| = \begin{pmatrix} n \\m \end{pmatrix} \bar d^{n-m}(i-\bar d+1)^m (i+1)^{n-m}\bar c^m \leq C_m i^n
\end{align}
for some constants $C_m$ (depending on $m,n,\bar c,\bar d$). Thus $\text{dim}\CF_i N \leq \sum_{m=0}^n |S_m| \leq \big(  \sum_{m=0}^n C_m\big) i^n$.
\phantom{x}\hfill $\square$

\subsubsection{The function spaces $\mathcal V_{n,m}$}

The cyclic modules relevant to this work arise from a particular representation of the $\W$ and $\E$ algebras. For any two non-negative integers $m$ and $n$, we define
\be \mathcal V_{n,m} = \{\text{functions}:\Z^m\to \mathbb V_n\}\,, \ee 
where $\mathbb V_n$ is the field of rational functions in $q^{\frac12},\{x_i^{\frac12}\}_{i=1}^n,\{z_{ij}\}_{i,j=1}^n$ as in \eqref{defVn}. We will think of $\mathcal V_{n,m}$ as a vector space over $k=\C(q)$.

The space $\mathcal V_{n,m}$ has a left action of $\E_{n+m}$ (and hence of its subalgebra $\W_{n+m}$) defined as follows.
Let us relabel the last $m$ pairs of generators of of $\E_{n+m}$ as $x_i,y_i \rightsquigarrow \hat x_{i-n},\hat y_{i-n}$  ($i>n$); thus
\be \E_{n+m} = k[x_1^\pm,y_1^\pm,...,x_n^\pm,y_n^\pm,\hat x_1^\pm,\hat y_1^\pm,...,\hat x_m^\pm,\hat y_m^\pm] /(\text{usual $q$-comm. rels.}) 
\,.\ee
The last $m$ pairs of generators have a familiar action 
\be \begin{array}{l} (\hat x_i^\pm \cdot f)(a_1,...,a_m) = q^{\pm a_i} f(a_1,...,a_m) \\[.1cm]
  (\hat y_i^\pm \cdot f)(a_1,...,a_m) = f(a_1,...,a_i\pm 1,...,a_m)\,. \end{array} \ee
The first $n$ pairs of generators have an action induced from that of $\E_n$ on the domain $\mathbb V_n$, which is given by
\be \label{E-Vn} \begin{array}{l} x_i^\pm :\,\text{multiplication by $x_i^\pm$} \\[.2cm]
 y_i^\pm\,: \big(x_j^{\frac12},z_{j\ell}\big) \mapsto \big(q^{\pm\frac12 \delta_{ij}} x_j^{\frac12}, q^{\frac12\delta_{ij}\delta_{i\ell}} x_j^{\pm\frac12\delta_{i\ell}}x_\ell^{\pm\frac12\delta_{ij}}z_{j\ell}\big)\,.
 \end{array} \ee
Explicitly, the induced action on $f\in \mathcal V_{n,m}$ is
\be \begin{array}{l} (x_i^\pm \cdot f)(a_1,...,a_m) \mapsto x_i^\pm  f(a_1,...,a_m)\,,\\[.1cm]
 (y_i^\pm \cdot f)(a_1,...,a_m) \mapsto  f(a_1,...,a_m)\Big|\raisebox{-.2cm}{$x_i\to q^\pm x_i\,,z_{ij}\to x_j^{\pm\frac12}z_{ij}\,(j\neq i),z_{ii}\to q^{\frac12}x_i^\pm z_{ii}$}\,. \end{array} \ee
It is straightforward to check that the $q$-commutation relations of $\E_{n+m}$ are respected by these combined actions.%
\footnote{\label{foot:alpha}A more intuitive way to understand the action of $\E_{n+m}$ on $\mathcal V_{n,m}$ is to fix $q$ to be a generic complex number and to set $x_i=q^{\alpha_i}$ and $z_{ij} =q^{\frac12 \alpha_i\alpha_j}$ for $\alpha_i\in\C$. Then 
the generators $x_i$ (resp. $\hat x_i$) of $\E_{n+m}$ act as multiplication by $q^{\alpha_i}$ (resp. $q^{a_i}$) and the generators $y_i$ (resp. $\hat y_i$) act by shifting $\alpha_i\mapsto \alpha_i+1$ (resp. $a_i\mapsto a_i+1$).
In particular, the somewhat awkward transformation of $z_{j\ell}$ in \eqref{E-Vn} is just that induced from a shift in $\alpha_i$.

Unfortunately, we will need to keep $q$ a formal algebraic variable (we cannot set it to a generic complex number) in order to gain control over the specialization to the ADO invariant later on.
}

Any function $f\in \mathcal V_{n,m}$ now generates a cyclic module for $\E_{n+m}$, denoted
\be M_f = \E_{n+m} f = \E_{n+m}\big/ \text{Ann}_{\E_{n+m}}(f)\,, \ee
Each $M_f$ is a submodule of the corresponding $\mathcal V_{n,m}$.
\begin{definition} \label{def:f-hol} We say that \emph{the function $f$ is $q$-holonomic} if the corresponding $\E_{n+m}$-module $M_f$ is $q$-holonomic.
\end{definition}

Similarly, $f$ generates a cyclic $\W_{n+m}$-module $N_f = \W_{n+m} f = \W_{n+m}\big/ \text{Ann}_{\W_{n+m}}(f)$. Note that such a $\W_{n+m}$-module can never have monomial torsion, because the $\W_{n+m}$ action on $\mathcal V_{n,m}$ extends to an $\E_{n+m}$ action, for which the generators $x_i,y_i,\hat x_i,\hat y_i$ are invertible.
By Prop.~\ref{prop:EW}, if $N_f$ is a $q$-holonomic $\W_{n+m}$-module, then $M_f$ is a $q$-holonomic $\E_{n+m}$-module.

\subsubsection{Examples}

We list some classic examples of $q$-holonomic functions $f\in \CV_{n,m}$ which will be useful in proving the $q$-holonomicity of $G_\mathbb{D}$. \medskip

\paragraph*{Constant and delta functions.}

The constant function $f(a_1,...,a_n) \equiv 1$ ($f\in \CV_{n,m}$) has annihilator ideal
\be \text{Ann}_{\E_{n+m}}^{}(f) = \E_{n+m}(y_i-1,\hat y_j-1)_{i=1}^n{}_{j=1}^m\,, \label{eg-const} \ee
and is $q$-holonomic by a straightforward application of Lemma \ref{lemma:hol-y}.

The delta function in discrete variables $h(a_1,...,a_n)=\delta_{a_1,0}\cdots \delta_{a_n,0} = \displaystyle\begin{cases} 1 & a_i=0\;\forall\,i \\ 0 & \text{otherwise}\end{cases}$\; has
\be \text{Ann}_{\E_{n+m}}^{}(h) = \E_{n+m}(y_i-1,\hat x_j-1)_{i=1}^n{}_{j=1}^m\,. \label{eg-delta-a} \ee
It is $q$-holonomic by an application of Lemma \ref{lemma:hol-y} with $\hat x_j\leftrightarrow \hat y_j$ swapped. (This ``swap,'' more precisely $(\hat x_j,\hat y_j)\mapsto (\hat y_j,\hat x_j^{-1})$, is an automorphism of $\E_{n+m}$ known as Mellin or Fourier transform, \emph{cf.} \cite[Sec. 1.3]{Sabbah}.)

One may also consider a cyclic $\E_{n+m}$-module $M=\E_{n+m} v$  with annihilator ideal
\be \text{Ann}_{\E_{n+m}}(v) =  \E_{n+m}(x_i-1,\hat y_j-1)_{i=1}^n{}_{j=1}^m\,. \label{eg-delta} \ee
It is $q$-holonomic, by Lemma \ref{lemma:hol-y} with $x_i\leftrightarrow y_i$ swapped.
This plays the role of the cyclic module generated by a delta-function in the \emph{continuous} variables, namely $f(a_1,...,a_n) = \delta(x_1-1)\cdots \delta(x_n-1)$. However, such a Dirac delta-function does not exist in our algebraic functional space $\mathbb V_n$, so the cyclic module $M=\E_{n+m} v$ is not embedded in $\CV_{n,m}$.

\medskip

\paragraph{Indicator functions.}

Generalizing the delta-function example above, the indicator function
\begin{subequations} \label{eg-theta}
\be \vartheta_{[a_2,a_3]}(a_1) = \begin{cases} 1 & a_2\leq a_1\leq a_3 \\ 0 & \text{else} \end{cases} \quad \in \, \CV_{0,3}  \label{theta-def} \ee
has the elements $(\hat x_1-q\hat x_3)(\hat y_3-1)$, $(\hat x_1-\hat x_2)(\hat y_2-1)$, and $(\hat x_1-\hat x_3)(\hat x_1-q^{-1}\hat x_2)(\hat y_1-1)$ in its annihilator ideal, and thus is $q$-holonomic for $\E_3$ by Lemma \ref{lemma:hol-y}. Its half-infinite cousin
\be \vartheta_{a_1\leq a_2} = \vartheta_{(-\infty,a_2]}(a_1) = \vartheta_{[a_1,\infty)}(a_2) := \begin{cases} 1 & a_1\leq a_2 \\
0 & \text{else} \end{cases} \quad \in \, \CV_{0,2} \ee
\end{subequations}
has annihilator ideal containing $(\hat x_2-q^{-1}\hat x_1)(\hat y_2-1)$ and $(\hat x_2-\hat x_1)(\hat y_1-1)$, and thus is $q$-holonomic for $\E_2$. Specializations of these functions to constant $a_1$ and/or $a_2$ and/or $a_3$ are similarly $q$-holonomic.

\medskip

\paragraph*{Linear exponentials.}

The ``linear'' functions
\be f(a_1)=q^{a_1}\quad (f\in\CV_{0,1})\,,\qquad g = x_1 \quad (g\in \CV_{1,0})\ee
are both $q$-holonomic, with annihilator ideals
\be \text{Ann}(f) = \E_1(\hat y_1-q)\,,\qquad \text{Ann}(g) = \E_1(y_1-q)\,.\ee
More generally, given any integer vectors $A=(A_1,...,A_n)\in (\Z^n)^*$, $\hat A=(\hat A_1,...,\hat A_m)\in (\Z^{m})^*$, the linear function
\be \label{eg-lin} f(a_1,...,a_m) = q^{\frac12 \hat A\cdot a}x^{\frac 12A} := q^{\frac12\hat A_1a_1+...+\frac12\hat A_ma_m} x_1^{\frac12 A_1}\cdots x_n^{\frac12 A_n} \qquad (f\in \CV_{n,m})\,. \ee
Its annihilator ideal in $\E_{n+m}$ has generators
\be \begin{cases} y_i-q^{\frac12A_i} & A_i\;\text{even} \\
  y_i^2 - q^{A_i} & A_i\;\text{odd} \end{cases} \;\;(i=1,...,n)\quad\text{and}\quad
  \begin{cases} \hat y_i-q^{\frac12\hat A_i} & \hat A_i\;\text{even} \\
  \hat y_i^2 - q^{\hat A_i} & \hat A_i\;\text{odd} \end{cases}\;\; (i=1,...,m)\,,
\ee
and thus $f$ is $q$-holonomic by a direct application of Lemma \ref{lemma:hol-y}.

\bigskip
\paragraph{Quadratic exponentials.}

The ``quadratic'' functions 
\be f(a_1) = q^{a_1^2}\quad (f\in \CV_{0,1})\,,\qquad g = z_{11}^2\quad (f\in \CV_{1,0}) \ee
are $q$-holonomic with annihilator ideals $\E_1(\hat y_1-q\hat x_1^2)$ and $\E_1(y_1-qx_1^2)$, respectively. Similarly, the ``quadratic'' function
\be f(a_1) = q^{\frac12 a_1^2}\quad (f\in \CV_{0,1})\,,\qquad g = z_{11}\quad (f\in \CV_{1,0}) \ee
are $q$-holonomic with annihilators $\E_1(\hat y_1^2-q^2\hat x_1^2)$ and $\E_1(y_1^2-q^2x_1^2)$. We may also consider mixed ``quadratic'' functions such as
\be f(a_1) = x_1^{a_1} \quad (f\in \CV_{1,1})\,,\ee
which is $q$-holonomic with annihilator $\text{Ann}(f) = \E_2(y_1-\hat x_1,\hat y_1-x_1)$.

More generally, let $\hat B:\Z^m\times \Z^m \to \Z$ and $B:\Z^n\times \Z^n\to \Z$ be symmetric bilinear forms, let $C:\Z^n\times \Z^m\to \Z$ be bilinear, and let $A\in (\Z^n)^*$, $\hat A\in (\Z^m)^*$ be integer vectors. Then 
\begin{align} \label{eg-quad} f(a) &= q^{\frac12 \hat B(a,a)+\frac12 \hat A\cdot a} x^{\frac12 C(-,a)+\frac12 A} z^B \qquad (f\in \CV_{n,m}) \\
 & := q^{\frac12 \sum_{ij} \hat B_{ij}a_ia_j+\frac12\sum_i \hat A_ia_i} \prod_{ij} x_i^{\frac12 C_{ij} a_j}\prod_i x_i^{\frac12 A_i} \prod_{ij} z_{ij}^{B_{ij}} \notag
\end{align}
is $q$-holonomic. Its annihilator ideal is cumbersome to write down in general form (because it depends on whether various parameters are even or odd), but easy to analyze. It is generated by expressions of the form $y_i-$(monomial in $q^\pm,x^\pm,\hat x^\pm$) or $y_i^2-$(monomial in $q^\pm,x^\pm,\hat x^\pm$), and by $\hat y_j-$(monomial in $q^\pm,x^\pm,\hat x^\pm$) or $\hat y_j^2-$(monomial in $q^\pm,x^\pm,\hat x^\pm$), for each $i=1,...,n$ and $j=1,...,m$. Thus being $q$-holonomic follows directly from Lemma \ref{lemma:hol-y}.

 \bigskip

\paragraph{Warning!}

The function $f(a_1) = q^{a_1^3}$ ($f\in \CV_{0,1}$) is well known not to be $q$-holonomic, \emph{cf.} \cite[Ex. 2.2]{GL-survey}. Similarly, the analogous ``cubic'' functions involving continous variables, such as $g(a_1) = x_1^{a_1^2}$ ($g\in \CV_{1,1}$) and $h(a_1)=z_{11}^{a_1}$ ($h\in \CV_{1,1}$) are not $q$-holonomic. \bigskip

\paragraph{$q$-Factorials}

Many types of $q$-factorials (or quantum dilogarithms) are $q$-holonomic. For $a\in \Z$, we recall the $q$-Pochhammer symbol \eqref{E:PochhammerSym} given by
\be \label{eg-qf} (x;q)_a := \begin{cases} (1-x)(1-qx)\cdots(1-q^{a-1}x) & a\geq 1 \\ 1 & a=0 \\ 0 & a \leq -1 \end{cases} \ee
This is an element of $\CV_{1,1}$ and its annihilator ideal contains%
\footnote{Note that the $(\hat x-q^{-1})$ factor in the first equation accounts for setting $(x;q)_a=0$ at negative values of $a$; at all positive $a$, the function  $(x;q)_a=0$ is simply annihilated by $\hat y+\hat x x-1$.} %
$(\hat x-q^{-1})(\hat y+\hat x x-1)$ and $(1-x)y+\hat x x-1$, whence by Lemma \ref{lemma:hol-y} it is $q$-holonomic as an $\E_2$-module.

Related $q$-holonomic functions from which we'll construct the R-matrix are
\begin{itemize}
\item $(x;q^2)_a \in \CV_{1,1}$, whose annihilator ideal contains $(\hat x-q^{-1})(\hat y+\hat x^2 x-1)$ and $(1-x)y+\hat x^2 x-1$\,;
\item $(q^{a_2};q^2)_{a_1} \in \CV_{0,2}$, whose annihilator ideal has $(\hat x_1-q^{-1})(\hat y_1+\hat x_1^2 \hat x_2-1)$, $(1-\hat x_2)\hat y_2+\hat x_1^2 \hat x_2-1$\,;
\item $\displaystyle \frac{1}{(x;q^2)_a}\in \CV_{1,1}$, whose annihilator ideal has $(\hat x-q^{-1})((1-\hat x^2 x)\hat y-1)$, $(1-\hat x^2 x)y+x-1$\,;
\item $\displaystyle \frac{1}{(q^2;q^2)_a}\in \CV_{0,1}$, whose annihilator ideal has $(\hat x-q^{-1})((1-q^2\hat x^2)\hat y-1)$\,.
\end{itemize}

\subsection{Closure properties}
\label{sec:closure}

A notable feature of $q$-holonomic modules is that they are closed under many algebraic operations. These closure properties enabled Garoufalidis and L\^e to efficiently prove that the colored Jones invariants of knots formed a $q$-holonomic family. They are of similar importance here.

We review some of the closure properties that will be used in the current work, as they apply to our functional spaces $\CV_{n,m}$ containing both discrete and ``continuous'' variables. Even though the initial application to Jones polynomials \cite{GL-J} only involved acting on functions of discrete variables, the closure properties themselves are much more general. They all derive from purely \emph{algebraic} properties of $q$-holonomic $\E_n$-modules, which make no reference to representations in a particular functional space. If one happens to be working with cyclic $\E_n$-modules generated by functions, the algebraic closure properties can simply be applied to that setting. (This perspective was also espoused in the recent survey \cite{GL-survey}.)

Thus, altogether, there is nothing mathematically novel in this section. We aim to illustrate how established closure properties apply in our setting of interest.

\begin{prop} \emph{\textbf{Closure properties}} \label{prop:closure} \smallskip
Suppose that $f,g \in \mathcal V_{n,m}$ are $q$-holonomic, with arguments $f(a) = f(a_1,\ldots,a_m)$.

\begin{subequations}
\noindent \emph{\textbf{a)}} (Addition and Multiplication)  The functions $f+g \in \CV_{n,m}$ and $fg\in \CV_{n,m}$ are $q$-holonomic. \smallskip

\noindent  \emph{\textbf{b)}} (Shifts) Choose vectors $c\in (k^*)^n$ and $d\in \Z^m$. Then
\be  f(a_1+d_1,...,a_m+d_m)\big|\raisebox{-.15cm}{$x_i\mapsto c_ix_i$ for $i=1,...,n$} \label{shift-f} \ee
is $q$-holonomic. \smallskip

\noindent \emph{\textbf{c)}} (Linear transformations) Let $A\in {\rm Mat}(n\times n',\Z)$, $C\in {\rm Mat}(n\times m',\Z)$ and $D\in {\rm Mat}(m\times m',\Z)$.

We define a $q$-holonomic function $h(a')\in \CV_{n',m'}$, given by
\be h(a')  :=   f(Da')\big|\raisebox{-.15cm}{$x\mapsto q^{Ca}x^A$} \label{lin-hol} \ee
Explicitly, the transformation of the $x$'s here is $x_i\mapsto  \prod_{j'=1}^{m'} q^{C_{ij'}a_{j'}} \prod_{i'=1}^{n'} x_{i'}^{A_{ii'}}$. Important special cases include specializations of discrete variables:
\be \text{$f(a_1,...,a_{m-1},a_m)\in \CV_{n,m}$ $q$-holonomic \;$\Rightarrow$\; $f(a_1,...,a_{m-1},0)\in\CV_{n,m-1}$ $q$-holonomic}\,; \ee
specializations in continuous variables:
\be  \text{$f(a)\in \CV_{n,m}$ $q$-holonomic \;$\Rightarrow$\; $f(a)\big|_{x_n=1}\in\CV_{n-1,m}$ $q$-holonomic}\,; \ee
and extensions in both sorts of variables: when $m\leq m'$ and $n\leq  n'$, we can view $f\in \CV_{n,m}$ as an element of $\CV_{n',m'}$ (a function independent of any extra $a$ or $x$ variables), and $f$ being $q$-holonomic for $\E_{n+m}$ implies that $f$ is $q$-holonomic for $\E_{n'+m'}$ as well. \smallskip

\noindent \emph{\textbf{d)}} The sum over a discrete variable
\be h(a_1,...,a_m,a_{m+1}) := \sum_{b=a_m}^{a_{m+1}} f(a_1,...,a_{m-1},b)\,,\qquad g\in \CV_{n,m+1} \label{summation-hol} \ee
is likewise $q$-holonomic. Similarly, when they converge, the half-infinite sums $\sum_{b=a_m}^\infty f(a_1,...,a_{m-1},b)$ and $\sum_{b=-\infty}^{a_m} f(a_1,...,a_{m-1},b)$ are $q$-holonomic functions in $\CV_{n,m}$; and  $\sum_{b=-\infty}^\infty f(a_1,...,a_{m-1},b)$ is $q$-holonomic in $\CV_{n,m-1}$.

\end{subequations}
\end{prop}

\noindent\emph{Proof.} The proofs of these statements are essentially identical to the arguments given in \cite{GL-J,GL-survey}, so we will be brief. 

For (a), let $M=\E_{n+m}f$ and $N=\E_{n+m}g$ be the modules generated by $f$ and $g$. The $\E_{n+m}$-module generated by the sum $f+g$ is a sub-quotient of the (algebraic) direct sum $M\oplus N$, and both sub-quotients and direct sums of $q$-holonomic modules are $q$-holonomic \cite{Sabbah}. Similarly, the $\E_{n+m}$-module generated by $fg$ is a submodule of the algebraic tensor product $M\otimes_{k[x^\pm]} N$; and tensor products of $q$-holonomic modules are $q$-holonomic \cite{Sabbah}.

For (b), we may simply note that for any $c\in (k^*)^n$ and $d\in \Z^m$, there is an automorphism of the algebra $\E_{n+m}$ given by $\gamma:(x_i,\hat x_j,y_i,\hat y_j)\mapsto (c_ix_i,q^{d_j}\hat x_j,y_i,\hat y_j)$, and a corresponding linear automorphism of $\CV_{n,m}$ sending
\be h(a_1,...,a_m)\quad\mapsto\quad h(a_1+d_1,...,a_m+d_m)\big|\raisebox{-.15cm}{$x_i\mapsto c_ix_i$ for $i=1,...,n$} \ee
as in \eqref{shift-f} that intertwines the automorphism $\gamma$ of the algebra. The property of being $q$-holonomic is preserved by any such automorphism.

For (c), we assemble $A,C,D$ into an $(n+m)\times (n'+m')$ matrix
\be U = \begin{pmatrix} A&0\\C& D \end{pmatrix}\,. \ee
This linear transformation defines a function $F:(\C^*)^{n+m}\to (\C^*)^{n'+m'}$ under which the pullback of coordinates is $F^* x_i = \prod_{j=1}^{n+m} x_j^{U_{ij}}$. This in turn induces an inverse image functor $F^!:\E_{m+n}$-mod$\to \E_{m'+n'}$-mod, which is shown in \cite[Sec. 2.3]{Sabbah} to preserve $q$-holonomic modules. Letting $M=\E_{n+m}f$, one finds that the module $N=\E_{n'+m'}h$ generated by the function in \eqref{lin-hol} is a sub-quotient of $F^!(M)$, and so $q$-holonomic.

For (d), we use the result of \cite[Sec 2.4]{Sabbah} that the algebraic convolution product of $q$-holonomic modules is $q$-holonomic. For any $h(a_1,...,a_m)$ and $h'(a_1,...,a_m)$, the function
\be h *_m h' := \sum_{b=-\infty}^\infty h(a_1,...,b+a_m) h'(a_1,...,-b)\quad \in \CV_{n,m}\,, \ee
when it exists, generates a submodule of the algebraic convolution product $(\E_{n+m}h)*(\E_{n+m}h')$, and thus is $q$-holonomic. Then we recall that indicator functions \eqref{theta-def} are $q$-holonomic. The summation given by \eqref{summation-hol} is obtained by convolving $f$ (extended to an element of $\CV_{n,m+2}$) with an indicator function; similarly, the half-infinite and infinite sums below \eqref{summation-hol} are obtained by convolving $f$ with half-infinite indicator functions and with the constant function, respectively. \\
\phantom{x}\hfill $\square$

\subsection{$G_{\mathbb D}$ is $q$-holonomic}
\label{sec:GDhol}

With the machinery of $q$-holonomic modules in place, we directly obtain

\begin{prop} \label{prop:GD}
 The ``diagram invariant'' $G_{\mathbb D}(r)$ defined in Section \ref{sec:GD}, which is an element of $\CV_{n,1}$, generates a $q$-holonomic module for $\E_{n+1}$.
\end{prop}

\noindent\emph{Proof.} All the individual functions \eqref{diag-piece-functions} associated to crossings, cups, and caps that get multiplied to define $G_{\mathbb D}^\times$ in \eqref{GDx} are $q$-holonomic in $\CV_{n,m+1}$. Specifically:
\begin{itemize}

\item The discrete delta-functions $\delta_{a,0} \in \CV_{0,1}$ that enter the final product $G_{\mathbb D}^\times$ are $q$-holonomic (Example \eqref{eg-delta-a}). We use Prop. \ref{prop:closure}c to trivially extend $\delta_{a,0}$ to a $q$-holonomic function in $\CV_{n,m+1}$, independent of the other discrete variables and all the $x$'s.

\item Consider the modified quantum dimension $\md[x_i] =  (-x_i)^{r-1} q^{\frac12 r(r+1)-1} \frac{1}{(q^4 x_i^2;q^2)_{r-1}} \in \CV_{1,1}$. This may be assembled as a product of 1) a $q$-factorial $\frac{1}{(x_i;q^2)_r}$, which was explained to be $q$-holonomic below \eqref{eg-qf}, and in which we use Prop. \ref{prop:closure}b to shift $x_i\to q^2 x_i$ and $r\to r-1$; 2) a general quadratic exponential $x_i^r q^{\frac12 r^2}$ as in Example \ref{eg-quad}, in which we shift $x_i\to -x_i$; 3) a linear exponential $x_i^{-1}q^{\frac12 r}$ as in Example \eqref{eg-lin}, in which we shift $x_i\to -x_i$; and 4) an overall constant $q^{-1}$. All these pieces are $q$-holonomic functions in $\CV_{1,1}$, so Prop. \ref{prop:closure}a guarantees their product will be $q$-holonomic as well. Then we use Prop. \ref{prop:closure}c to extend $\md[x_i]$ to a $q$-holonomic function in $\CV_{n,m+1}$ (independent of the other $x$'s and discrete $a,b,c$ variables).

\item The cup and cap functions $\eta_a[x_i]=1,\epsilon_a[x_i]=1$, treated as elements of $\CV_{n,m+1}$, are just constant functions independent of all the variables; they are $q$-holonomic by Example \ref{eg-const}.

The cup and cap functions $\eta^*_a[x_i]=q^{2a(1-r}x_i^{r-1}$ and $\epsilon^*_a[x_i]=q^{-2a}x_i$, both in $\CV_{1,2}$ (the discrete variables are $a$ and $r$), are products of general linear and quadratic exponentials, as in Examples \ref{eg-lin}, \ref{eg-quad}. They are extended by Prop. \ref{prop:closure}c to $q$-holonomic functions in $\CV_{n,m+1}$.

\item The R-matrices $R^{a,b}_{c,d}[x_i,x_j],(R^{-1})^{a,b}_{c,d}[x_i,x_j]\in \CV_{2,5}$ are products of discrete delta-functions (Example \eqref{eg-delta-a}), indicator functions (Example \eqref{eg-theta}), linear and quadratic exponentials (Examples \ref{eg-lin}, \ref{eg-quad}), and $q$-factorials (Example \eqref{eg-qf}), all with various shifts (Prop. \ref{prop:closure}b) and linear transformations (Prop. \ref{prop:closure}c).
Some of the $q$-factorials involve $q^{-2}$ rather than $q^2$; but it is easy to put them into the same form as Example \eqref{eg-qf} by observing that
\be (y;q^{-2})_a =  (-y)^aq^{-a(a-1)} (y^{-1};q^2)_a\,,\ee
which is a ``standard'' $q$-factorial multiplied by linear and quadratic exponentials. Thus $R,R^{-1} \in \CV_{2,5}$ are $q$-holonomic.  We use Prop. \ref{prop:closure}c to extend $R,R^{-1}$ to $q$-holonomic functions in $\CV_{n,m+1}$.

\end{itemize}

\noindent The product of $G_{\mathbb D}^\times \in \CV_{n,m+1}$ of these $q$-holonomic functions is $q$-holonomic by Prop. \ref{prop:closure}a. The final diagram invariant $G_{\mathbb D}\in \CV_{n,1}$ is obtained from $G_{\mathbb D}^\times$ by summing over every discrete variable, and then specializing the bounds of each summation to be $0$ and $r-1$. It is therefore $q$-holonomic by Prop. \ref{prop:closure}d (for the summations) and Prop. \ref{prop:closure}c (for the specializations). \hfill $\square$


\section{Specializing to a root of unity}
\label{sec:spec}

We proved in Proposition \ref{prop:GD} that, for any $(1,1)$-tangle diagram $\mathbb D$, $G_{\mathbb D}$ is $q$-holonomic. Thus it generates a $q$-holonomic module for $\E_{n+1} = \C(q)[x_1^\pm,y_1^\pm,...,x_n^\pm,y_n^\pm,\hat x^\pm,\hat y^\pm]$, where the action on functions $f\in \CV_{n,1}$ (including $G_{\mathbb D}$) is
\be \label{En1-action} \begin{array}{l@{\qquad}l} x_i \,:\, f\mapsto x_if & \hat x\,:\, f \mapsto q^r f \\
y_i \,:\, f \mapsto f\big|\raisebox{-.15cm}{$x_i\mapsto q x_i, z_{ii}\mapsto q^{\frac12}x_i z_{ii},z_{ij}\mapsto x_j^{\frac12}z_{ij}$} & \hat y\,:\, f \mapsto  f\big|_{r\,\mapsto\, r+1}\,. \end{array} \ee
We would now like to prove that the ADO invariant $N^r_L$ is $q$-holonomic, in an appropriate sense. The ADO invariant is an actual invariant of the framed, oriented link $L$ obtained by closing the tangle with diagram $\mathbb D$.

We recall from Section \ref{sec:GD} that the ADO invariant is obtained from $G_{\mathbb D}$ by setting $q^{\frac12}=\rtu^{\frac12},$ $x_i^{\frac12}=\rtu^{\alpha_i/2}$, and $z_{ij}=\rtu^{\alpha_i\alpha_j/}$.  More succinctly, if we make explicit the dependence on $x,z,q$ in $G_{\mathbb D}(r,x^{\frac12},z;q)$, then
\be \label{ADO-spec3} N_L^r(\alpha) = G_{\mathbb D}(r;\rtu^\alpha/2,\rtu^{\alpha\otimes \alpha/2};\rtu^{1/2})\,. \ee

As prefaced in the introduction, explaining what it means for functions defined at roots unity $q=\zeta_{2r}$ to be holonomic is a subtle matter. 
By Corollary \ref{cor:struc}, we may think of the ADO invariant at each fixed $r$ as an element of the functional space
\be  N_L^r \in \CV_n^{(r)} := \{\text{quasi-periodic, meromorphic functions}: (\C/2r\Z)^n\to \C\}\,, \ee
with periodicity of the form $f(\alpha_1,...,\alpha_i+2r,...,\alpha_n)=\zeta_{2r}^{\sum_j 2rC_{ij}\alpha_j}f(\alpha_1,...,\alpha_i,...,\alpha_n)$ for some (unspecified) $C_{ij}$.
Each space $\CV_n^{(r)}$ has an action of the $q$-Weyl algebra at a $2r$-th root of unity
\be \label{Enr} \mathcal E_n^{(r)} := \C[x_1^\pm,y_1^\pm,...,x_n^\pm,y_n^\pm]/(y_ix_j - \zeta_{2r}^{\delta_{ij}}x_jy_i) \ee
given by
\be \label{Ealpha} x_i\cdot f(\alpha) = \zeta_{2r}^{\alpha_i} f(\alpha)\,,\qquad y_i\cdot f(\alpha) = f(\alpha_1,...,\alpha_i+1,...,\alpha_n) \qquad \big(f\in \CV_n^{(r)}\big)\,. \ee
However, due to the quasi-periodicity in Part (iii) of Cor. \ref{cor:struc}, it is also clear that at each fixed $r$ the ADO invariant of an $n$-strand link will trivially satisfy $n$ independent recursion relations  $\big(\prod_j x_j^{-2rC_{ij}}y_i^{2r}-1\big)N_L^r=0$ ($i=1,...,n$), where $C_{ij}$ is the linking matrix of $L$. In order to obtain a nontrivial statement, we work in a \emph{family}, considering all $r\in \mathbb N_{\geq 2}$ at once.

Consider the evaluation maps
\be \text{ev}_r: \begin{array}{ccc}\E_n &\dashrightarrow& \mathcal E_n^{(r)} \\
    A(x,y;q) &\mapsto &A(x,y;\zeta_{2r})\,. \end{array} \ee
Note that each individual $\text{ev}_r$ is not defined on all of $\E_n= \C(q)[x_1^\pm,y_1^\pm,...,x_n^\pm,y_n^\pm]/(y_ix_j - q^{\delta_{ij}}x_jy_i)$, since elements of $\E_n$ may have denominators in $q$ that vanish at $q=\zeta_{2r}$. However, given any $A\in \E_n$, the family of evaluations $\big\{\text{ev}_rA\big\}_{r\in \mathbb N}$ is defined for all but finitely many $r\in \mathbb N$. Moreover, where it makes sense, $\text{ev}_r$ is clearly an algebra map, satisfying $\text{ev}_r(AB)=\text{ev}_r(A)\text{ev}_r(B)$.

For any family of functions $\big\{f_r\in \CV_n^{(r)}\big\}_{r\in \mathbb N}$, we may construct a left ideal $\mathcal I[f] \subseteq \E_n$ as
\be \label{def-If} \mathcal I[f] := \{ A\in \E_n\,|\, \text{ev}_r(A) f_r=0\;\text{for all but finitely many $r\in \mathbb N$}\}\,, \ee
throwing out any $r$'s for which $\text{ev}_r(A)$ is not defined. Then we say:

\begin{definition} \label{def:qfam}
The family of functions $\big\{f_r\in \CV_n^{(r)}\}_{r\in \mathbb N}$ is $q$-holonomic if the associated cyclic module $\E_n/\mathcal I[f]$ is $q$-holonomic.
\end{definition}

We will prove in this section that the \emph{family} of ADO invariants $\big\{N_L^r \in \mathcal V_n^{(r)}\big\}_{r\geq 2}$ of any framed, oriented link $L$ is $q$-holonomic. We will also prove that the associated ideal $\mathcal I[N_L]$ is contained in the annihilation ideal of the colored Jones polynomial of $L$.

\subsection{Quantum Hamiltonian reduction}
\label{sec:red}

We introduce a preliminary result that will help us relate the annihilation ideal of $G_{\mathbb D}$ and the family of ADO invariants. The result is purely algebraic in nature, independent of particular functional spaces.

Suppose we have a left ideal $\mathcal I_n\subseteq \E_n$ and a nonzero element $c\in k^*=\C(q)^*$. Then we can construct a left ideal $\mathcal I^c_{n-1}\subseteq \E_{n-1}$ by first taking the intersection of $\mathcal I_n$ with the subalgebra
\be \label{def-Etilde} \widetilde \E_{n-1} := k[x_1^\pm,y_1^\pm,...,x_{n-1}^\pm,y_{n-1}^\pm,x_n^\pm]/(y_ix_j-q^{\delta_{ij}}x_jy_i)_{i,j=1}^{n-1} \;\; \subset \;\E_n\,, \ee
in which $x_n$ is central (because $y_n$ is no longer present), and then specializing $x_n=c$, noting that $\E_{n-1}\simeq \widetilde \E_{n-1}/(x_n-c)$.
 All together,
\be \mathcal I_{n-1}^c = \big(\mathcal I_n \cap \widetilde \E_{n-1}\big)\big|_{x_n=c}\,. \label{def-Inc} \ee
Explicitly, the elements of $\mathcal I_n$ and $\mathcal I_{n-1}^c$ are related by
\be A(x_1,y_1,...,x_{n-1},y_{n-1})\in \mathcal I_{n-1}^c \quad\Leftrightarrow \quad
\begin{array}{c} \exists\;\; \widetilde A(x_1,y_1,...,x_n) \in \mathcal I_n\;\text{independent of $y_n$} \\[.1cm]
\text{s.t. $\widetilde A(x_1,y_1,...x_{n-1},y_{n-1},c)=A$ }\end{array}  \ee

The relation between the associated modules $\E_n/\mathcal I_n$ and $\E_{n-1}/\mathcal I_{n-1}^c$, is a version of quantum Hamiltonian reduction. In this case, the reduction is with respect to a multiplicative moment map $x_n$, and central character $c$.%
\footnote{Very similar reductions were used in \cite{Dimofte-QRS} to construct quantum A-polynomials from ideal triangulations of knot complements. The construction there was not yet rigorous, but could hopefully be made so using Prop.~\ref{prop:red}.} %
Quantum Hamiltonian reduction is a familiar operation in the study of D-modules and representation theory, \emph{cf.}  \cite{EtingofGinzburg, CBEG, Losev, Jordan}, which is generally expected to preserve holonomic modules (since it is the quantization of a Lagrangian correspondence). We will use the following result:

\begin{prop} \label{prop:red}
For $n\geq 2$, let $\mathcal I_n\subseteq \E_n$ be a left ideal, and let $\mathcal I_{n-1}^c = \big(\mathcal I\cap \widetilde \E_{n-1}\big)\big|_{x_n=c}\subseteq \E_{n-1}$ as above. If $\E_n/\mathcal I_n$ is a $q$-holonomic $\E_n$-module then $\E_{n-1}/\mathcal I_{n-1}^c$ is a $q$-holonomic $\E_{n-1}$-module.
\end{prop}

We give a self-contained proof of this Proposition in Appendix \ref{app:red}.

A useful way to relate Hamiltonian reduction to more elementary operations on $q$-holonomic modules is the following. Let $v$ be the generator of $\E_n/\mathcal I_n$ and let $\delta^{(n)}_c$ denote the generator of the module $\E_n/(y_1-1,...,y_{n-1}-1,x_n-c)$, a ``delta-function'' module in the final variable $x_n$. Just like Example \eqref{eg-delta}, this delta-function module is $q$-holonomic. We denote by $\E_n(v\otimes \delta^{(n)}_c)$ the submodule of the tensor-product-module $(\E_n v)\otimes (\E_n \delta^{(n)}_c)$ generated by $v\otimes \delta^{(n)}_c$.

Let us also consider the map of rings $f^*:k[x_1^\pm,...,x_n^\pm]\mapsto k[x_1^\pm,...,x_{n-1}^\pm]$ given by $f^*(x_i)=x_i$ for $1\leq i \leq n-1$ and $f^*(x_n)=c$. There is a corresponding inverse-image functor $f^!:\E_n\text{-mod}\to\E_{n-1}\text{-mod}$ defined in \cite[Sec. 2.3]{Sabbah}. It is explained in the proof of Prop. \ref{prop:red} in Appendix \ref{app:red} that
\be \E_{n-1}/\mathcal I_{n-1}^c \simeq f^!\big(\E_{n}(v\otimes \delta^{(n)}_c)\big)\,. \ee
Once one realizes this, it follows from the fact that tensor products, subs, and inverse images all preserve $q$-holonomic modules that $\E_{n-1}/\mathcal I_{n-1}^c$ must be $q$-holonomic as well.

\medskip

We note that the quantum Hamiltonian reduction discussed above is closely related to specialization of variables, in the case of cyclic $\E_n$-modules generated by functions. For example, if $\E_n$ acts on some space of functions of $(x_1,...,x_n)$, and the function $f(x_1,...,x_n)$ generates a cyclic module $\E_n f = \E_n/\mathcal I_n$, $\mathcal I_n = \text{Ann}_{\E_n}(f)$, then it is easy to see that the specialization $f_c(x_1,...,x_{n-1}) := f(x_1,...,x_{n-1},c)$ generates a module $\E_{n-1} f_c= \E_{n-1}/\mathcal I_{n-1}$, such that the Hamiltonian-reduction ideal $\mathcal I_{n-1}^c$ above satisfies $\mathcal I_{n-1}^c \subseteq \mathcal I_{n-1}$. In other words, the specialized module $\E_{n-1} f_c$ is a quotient of $\E_{n-1}/\mathcal I_{n-1}^c$. Thus, a corollary of Prop. \ref{prop:red} is that when $f$ is $q$-holonomic its specialization $f_c$ must be $q$-holonomic as well. Of course, we already knew this (Prop. \ref{prop:closure}c). The virtue of the algebraic formulation of quantum Hamiltonian reduction above is that it applies even when considering modules that are not generated by functions; that is how we will use it in the next section.

\subsection{The ADO invariants are a $q$-holonomic family}
\label{sec:ADOhol}

We are now ready to prove one of our main results, by using quantum Hamiltonian reduction to implement the specializations $q^r=-1$ in the ADO invariants.

\begin{theorem} \label{thm:ADO}
Let $L$ be a framed, oriented link with $n$ components. Then the family of ADO invariants $\{N_L^r\}_{r\geq 2}$ is $q$-holonomic for $\E_n$. In other words, the associated ideal
\be \mathcal I[N_L] := \{A\in \E_n \,|\, \text{ev}_r(A) N_L^r=0\;\text{for all but finitely many $r$}\} \ee
as in \eqref{def-If} defines a $q$-holonomic module $\E_n/\mathcal I[N_L]$.
\end{theorem}

\noindent\emph{Proof}. Choose a diagram $\mathbb D$ of a $(1,1)$ tangle whose closure is $L$, as in Section \ref{sec:preADO}, and let $G_{\mathbb D}(r;x^{\frac12},x;q^{\frac12})\in \CV_{n,1}$ be the associated ``diagram invariant.'' From Proposition \ref{prop:GD}, we know that $G_{\mathbb D}$ generates a $q$-holonomic left $\E_{n+1}$-module (with $\E_{n+1} = k[x_1^\pm,y_1^\pm,...,x_n^\pm,y_n^\pm,\hat x^\pm,\hat y^\pm]/(...)$ acting as in \eqref{En1-action}). Let
\be \mathcal I_{n+1} = \text{Ann}_{\E_{n+1}} G_{\mathbb D} \ee
be its annihilation ideal. Construct the reduced ideal
\be \mathcal I_{n}^{-1} := \big(\mathcal I_{n+1}\cap \widetilde \E_n\big)\big|_{\hat x = -1} \quad \subseteq \E_n\ee
as in \eqref{def-Inc}. This is quantum Hamiltonian reduction at $c=-1$ eliminates the $\hat y$ variable (which shifted $r\mapsto r+1$) and sets the  $\hat x$ variable (which acted as $q^r$) to $-1$. 

We claim that $\mathcal I_n^{-1}\subseteq \mathcal I[N_L]$. To see this, choose any $A(x,y;q)=A(x_1,y_1,...,x_n,y_n;q)\in \mathcal I_n^{-1}$. By the definition of $\mathcal I_n^{-1}$, there exists $\tilde A(x,\hat x,y;q)\in \widetilde\E_n = k[x_1,y_1,...,x_n,y_n,\hat x]\subset \E_{n+1}$ such that
\be \label{AtA} \tilde  A(x,\hat x=-1,y;q)= A(x,y;q) \qquad
\text{and}\qquad \tilde A(x,\hat x,y;q) G_{\mathbb D}(r;x^{\frac12},z;q)\,.  \ee
Choose any nonzero $f(q)\in\C[q]$ such that $f(q)A(x,y;q)$ and $f(q)A(x,\hat x,y;q)$ both have evaluations at $q=\zeta_{2r}$ for all $r\in \mathbb N$.
From the first equality in \eqref{AtA}, we have $f(\zeta_{2r})\tilde A(x,\hat x,y;\zeta_{2r})=f(\zeta_{2r})A(x,y;\zeta_{2r})$ for all $r$. Combining this with the second equality in \eqref{AtA}, evaluated at $q=\zeta_{2r}$, we have
\be f(\zeta_{2r}) A(x,y;\zeta_{2r}) G_{\mathbb D}(r;x^{\frac12},z;\zeta_{2r}) = 0\,. \ee
We may further specialize $x^{\frac12}= e^{\frac{i\pi}{2r}\alpha}$ and $z=e^{\frac{i\pi}{2r}\alpha\otimes \alpha}$ as in \eqref{ADO-spec3}, leading to
\be f(\zeta_{2r}) A(x,y;\zeta_{2r}) N_L^r(\alpha) = 0 \ee
for all $r$, with action \eqref{Ealpha}. Since $f(\zeta_{2r})$ can only vanish at (at most) finitely many values of $r\in \mathbb N$, we find that $A(x,y;\zeta_{2r})\in \mathcal I[N_L]$.

From Proposition \ref{prop:red} we know that the module $\E_n/\mathcal I_n^{-1}$ is $q$-holonomic. Moreover, since $\mathcal I_n^{-1}\subseteq \mathcal I[N_L]$, we find that  $\E_n/\mathcal I[N_L] \simeq (\E_n/\mathcal I_n^{-1})/(\mathcal I[N_L]/\mathcal I_n^{-1})$ is a quotient of $\E_n/\mathcal I_n^{-1}$. Since quotients of $q$-holonomic modules are $q$-holonomic by \cite[Cor. 2.1.6]{Sabbah}, it follows that $\E_n/\mathcal I[N_L]$ is $q$-holonomic. \phantom{X} \hfill $\square$

\subsection{Relation to the AJ conjecture}
\label{sec:Jones}

Finally, we can relate the recursion relations satisfied by the ADO family to those satisfied by the colored Jones function. Let $L=K$ be an oriented knot with framing $f$.

\begin{theorem} \label{thm:Jones}
Let $\mathcal I[N_K]\in \E_1$ be the ideal in Theorem \ref{thm:ADO} that annihilates the ADO family. Let $\big(J_N(q)\big)_{N\in \mathbb N}$ be the sequence of colored Jones polynomials of $K$. Then for every element $A(x,y;q)\in \mathcal I[N_K]$ we have
\be A(q^{-1}x,(-1)^{f+1}y;q)\, J_N(q) = 0\,,\ee
where $x$ acts as multiplication by $q^N$ and $y$ acts by shifting $N\mapsto N+1$.
\end{theorem}

\noindent\emph{Proof.} From Corollary \ref{cor:struc} we see that the poles in $N_K^r(\alpha)$ come entirely from the denominator $\rtu^{r\alpha}-\rtu^{-r\alpha}$ in the modified quantum dimensions. Then, noting that
\be \text{Res}_{\alpha=N-1}\frac{1}{\rtu^{r\alpha}-\rtu^{-r\alpha}} = \text{Res}_{\alpha=N-1} \frac{1}{2i\sin(\pi\alpha)} = \frac{(-1)^{N-1}}{2\pi i}\,, \ee
we may rewrite Proposition \ref{prop:Jones} to say that
\be \label{J-nonres} (\zeta_{2r}^\alpha-\zeta_{2r}^{-\alpha})N_K^r(\alpha) \big|_{\alpha=N-1} = C_r (-1)^{f(N-1)} J_N(\zeta_{2r})\,, \ee
where $C_r=2 i^{-r}\sin\frac{\pi}{r}$ is a constant that depends on $r$ but not on $N$.

Also note that with $y$ acting as a shift $\alpha\mapsto \alpha+1$ we have $y(\zeta_{2r}^\alpha-\zeta_{2r}^{-\alpha})= (\zeta_{2r}^\alpha-\zeta_{2r}^{-\alpha})(-y)$, and so
\be \label{AJsign} A(x,y;\zeta_{2r})N_K^r(\alpha) = 0 \quad\Leftrightarrow\quad A(x,-y;\zeta_{2r})(\zeta_{2r}^\alpha-\zeta_{2r}^{-\alpha})N_K^r(\alpha) =0\,. \ee
Similarly, with $y$ acting as a shift $N\mapsto N+1$ we have $y(-1)^N = (-1)^N(-y)$,  so
\be \label{AJsign2} A(x,y;\zeta_{2r}) (-1)^{f(N-1)}J_N(\zeta_{2r}) = 0  \quad\Leftrightarrow\quad  A(x,(-1)^{f}y;\zeta_{2r}) J_N(\zeta_{2r})=0\,.\ee
Let $A(x,y;q)$ be any element of the ideal $\mathcal I[N_K]$. For every value of $r$ such that $A(x,y;q)$ is nonsingular at $q=\zeta_{2r}$ we have $A(x,y;\zeta_{2r})N_K(\alpha)=0$; and then from \eqref{J-nonres}--\eqref{AJsign2} we obtain
\be A(q^{-1}x,(-1)^{f+1}y;\zeta_{2r})J_N(\zeta^{2r}) = 0\,.\label{AJr} \ee
(The extra shift $x\to q^{-1}x$ is made to ensure that $x$ acting as $q^\alpha$ on the ADO is compatible with $x$ acting as $q^{N}$ (rather than $q^{N-1}$) on the colored Jones.)
Now consider the functions
\be B_n(q) := A(q^{-1}x,(-1)^{f+1}y;q)J_N(q) \in \C(q)\,, \qquad n\in\mathbb N \ee
Due to \eqref{AJr}, each rational function $B_N(q)$ has zeroes at an infinite set of distinct points $q=\zeta_{2r}$.  (Note: there are at most finitely many poles in $B_N(q)$, and if they occur at roots of unity, the corresponding values of $r$ may be thrown out without affecting this argument.)
 Each function $B_N(q)$ must therefore be identically zero. \phantom{X}\hfill $\square$

We have shown that, up to an algebra automorphism that rescales $(x,y)\mapsto(q^{-1}x,(-1)^{f+1}y)$, the annihilation ideal $\mathcal I[N_K]$ of the ADO family is included in the annihilation ideal of the colored Jones function. If we further assume the AJ Conjecture of \cite{Gar-AJ} (with a physical origin in \cite{Gukov}), it follows that:

\begin{cor} \label{cor:AJ}
(Assuming the AJ Conjecture of \cite{Gar-AJ}.) Let $K$ be a knot with framing $f$ and let $A(x,y;q)$ be any element of the ADO ideal $\mathcal I[N_K]$ that admits evaluation at $q=1$. Then $A(m,(-1)^{f+1}\ell;1)$ is divisible by the A-polynomial $\mathbf A(m,\ell)$ of $K$.
\end{cor}

\begin{remark}
Theorem 66 of the upcoming revised version of \cite{Willetts} proves the converse to our Theorem \ref{thm:Jones}: that the colored Jones annihilation ideal is included in the ADO annihilation ideal. Taken together, these results imply that the two annihilation ideals are isomorphic.
Our computations in Appendix \ref{app:comp} confirm this isomorphism.
\end{remark}

\appendix

\section{Proof of Proposition \ref{prop:red}}
\label{app:red}

We give here an elementary proof of Proposition \ref{prop:red}, on quantum Hamiltonian reduction. We use the same notation as in Section \ref{sec:red}. The result we are aiming for is: \medskip

\emph{
\noindent For $n\geq 2$, let $\mathcal I_n\subseteq \E_n$ be a left ideal, and let $\mathcal I_{n-1}^c = \big(\mathcal I\cap \widetilde \E_{n-1}\big)\big|_{x_n=c}\subseteq \E_{n-1}$ as in \eqref{def-Inc}. If $\E_n/\mathcal I_n$ is a $q$-holonomic $\E_n$-module then $\E_{n-1}/\mathcal I_{n-1}^c$ is a $q$-holonomic $\E_{n-1}$-module. 
} \bigskip

Without loss of generality, we may assume $c=1$. Otherwise we may use the automorphism of $\E_n$ given by
\be (x_i,y_i)\mapsto \begin{cases} (x_i,y_i) & i\leq n-1 \\ (cx_i,y_i) & i=n \end{cases} \ee
to intertwine the reduction at $x_n=c$ with reduction at $x_n=1$.

Let $v$ denote the generator of $\E_n/\mathcal I_n$, whose annihilation ideal is $\mathcal I_n$. Let us also denote by
\be  \E_{n-1}=k[x_1^\pm,y_1^\pm,...,x_{n-1}^\pm,y_{n-1}^\pm]/(y_ix_j-q^{\delta_{ij}}x_jy_i)\,,\qquad \E_1 =k[x_n^\pm,y_n^\pm]/(y_nx_n-qx_ny_n) \ee
the standard $q$-Weyl algebra in the first $n-1$ pairs of variables and the last pair, respectively; and let us introduce the ``delta-function'' module
\be M^\delta_1= \E_1/(x_n-1) = \E_1 \delta^{(1)}\,, \ee
with formal generator $\delta^{(1)}$ satisfying $(x_n-1)\delta^{(1)}=0$, and its extension to an $\E_n$-module
\be  M^\delta_n = \E_n/(y_1-1,...,y_{n-1}-1,x_n-1) = \E_n \delta^{(n)} \ee
with formal generator $\delta^{(n)}$ satisfying $(y_i-1)\delta^{(n)}=0$ for $i=1,...,n-1$ and $(x_n-1)\delta^{(n)}=0$. Both $M_1^\delta$ and $M_n^\delta$ are $q$-holonomic (for $\E_1$ and $\E_n$, respectively), as in Example \eqref{eg-delta}.

It is also useful to recall that the \emph{tensor product} of $\E_n$-modules $U\otimes W$ has underlying vector space $U\otimes_{k[x^\pm]} W$ and action $x_i^\pm(u\otimes w) := (x_i^\pm u)\otimes  w = u\otimes (x_i^\pm w)$, $y_i^\pm(u\otimes w) :=(y_i^\pm u)\otimes(y_i^\pm w)$. In contrast, the \emph{exterior product} of an $\E_{n-1}$-module $U$ and an $\E_1$-module $W$ is defined to have underlying vector space $U\otimes_kW$ and action $x_i^\pm(u\otimes w)=(x_i^\pm u)\otimes w$, $y_i^\pm(u\otimes w)=(y_i^\pm u)\otimes w$ for $i\leq n-1$ and $x_n^\pm(u\otimes w) = u\otimes(x_n^\pm w)$, $y_n^\pm(u\otimes w) = u\otimes(y_n^\pm w)$. A special case is the exterior product of the algebras themselves, $\E_{n-1}\boxtimes \E_1\simeq \E_n$.

Now let $\widetilde M$ denote the submodule of the tensor product $(\E_n/\mathcal I_n)\otimes M^\delta_n$ generated by $v\otimes \delta^{(n)}$,
\be \widetilde M = \E_n(v\otimes \delta^{(n)})\,. \ee
$\widetilde M$ is $q$-holonomic because $q$-holonomic modules are closed under taking tensor products and subs (Section \ref{sec:closure}, \cite[Cor. 2.1.6, Prop 2.4.1]{Sabbah}). We will show that

\begin{lemma} \label{lemma:box}
$\widetilde M$ decomposes as an exterior product of $\E_{n-1}$ and $\E_1$ modules
\be \label{M-decomp} \widetilde M \simeq (\E_{n-1}/\mathcal I_{n-1}^1)\boxtimes M_1^\delta\,, \ee
whose first factor is precisely the module $\E_{n-1}/\mathcal I_{n-1}^1$ in the statement of Prop. \ref{prop:red}.
\end{lemma}

\noindent\emph{Proof of Prop. \ref{prop:red}}.

Assuming the Lemma, the most efficient way to prove the proposition is to consider the map
\be f:(\C^*)^{n-1} \to (\C^*)^{n}\,,\qquad f(x_1,...,x_{n-1}) = (x_1,...,x_{n-1},1) \ee
and to apply the associated inverse image functor $f^!$ to $\widetilde M$. Explicitly, the inverse image functor $ f^!:\E_n\text{-mod} \to \E_{n-1}\text{-mod}$ acts on an $\E_n$-module $U$ by tensoring it over $\E_n$ with the $(\E_{n-1},\E_n)$ bimodule
\be \mathcal E := k[x_1^\pm,...,x_n^\pm]/(x_n-1)\underset{k[x_1^\pm,...,x_n^\pm]}\otimes \E_n  \simeq  (x_n-1)\E_n\big\backslash \E_n\,.  \ee
Thus in general $f^! U := \mathcal E\otimes_{\E_n} U \simeq (x_n-1)U\big\backslash U$. In the case of the product $\widetilde M=(\E_{n-1}/\mathcal I_{n-1}^1)\boxtimes M_1^\delta$, the inverse image functor just removes the $M_1^\delta$ factor, giving
\be f^!\widetilde M = \E_{n-1}/\mathcal I_{n-1}^1  \,. \ee
Since inverse image (the zeroth cohomology of the derived inverse image of \cite[Prop. 2.3.2]{Sabbah}) preserves $q$-holonomic modules, $ \E_{n-1}/\mathcal I_{n-1}^1$ must be $q$-holonomic. \phantom{X}\hfill $\square$

\bigskip

An alternative proof that $\widetilde M=(\E_{n-1}/\mathcal I_{n-1}^1)\boxtimes M_1^\delta$ being $q$-holonomic implies that $\E_{n-1}/\mathcal I_{n-1}^1$ is $q$-holonomic comes from comparing to $\W_n$-modules. We include this for completeness.

Denote by $w$ the generator of $\E_{n-1}/\mathcal I_{n-1}^1$ (whose annihilation ideal is $\mathcal I_{n-1}$), and let $N_{n-1}=\W_{n-1} w = \W_{n-1}/(\mathcal I_{n-1}^1\cap \W_{n-1})$. The canonical good filtration on this module is given by
\be \mathcal F_i N_{n-1} = \{\beta w\,|\, \text{deg}_{x,y} \beta\leq i\}\,, \ee
where $\text{deg}_{x,y}$ denotes total degree in $x_1,...,x_{n-1}$ and $y_1,...,y_{n-1}$.
Let $d_i = \text{dim}_k \mathcal F_i N_{n-1}$. 

Similarly, let $\widetilde N = \W_n(w\boxtimes \delta^{(1)})$. By \cite[Prop. 3.4]{GL-survey}, $\widetilde{N}$ is a $q$-holonomic $\mathbb{W}_n$ module. The canonical good filtration on $\widetilde N$ is given by $\mathcal F_i \widetilde N = \{\beta w\,|\, \text{deg}_{x,y} \beta\leq i\}$, where $\text{deg}_{x,y}$ denotes total degree in $x_1,...,x_{n}$ and $y_1,...,y_{n}$. Let $\tilde d_i =  \text{dim}_k \mathcal F_i \widetilde N$. Due to the product structure
\be \widetilde N = \W_n(w\boxtimes \delta^{(1)}) = (\W_{n-1}w)\otimes_k (\W_1\delta^{(1)}) \simeq (\W_{n-1}w)\otimes_k \C[y_n] \ee
we find that $\mathcal F_i\widetilde N \simeq \bigoplus_{j=0}^i \mathcal F_{i-j} N_{n-1}\otimes y_n^j$; so $\tilde d_i = \sum_{j=0}^i d_j$, or equivalently $d_i = \tilde d_i-\tilde d_{i-1}$. Since $\widetilde N$ is $q$-holonomic, there is a polynomial $s(i)$ of degree $n$ such that $\tilde d_i=s(i)$ for all sufficiently large $i$. Therefore, $d_i = s(i)-s(i-1)$ is a polynomial of degree $n-1$ for all sufficiently large $i$, whence $N_{n-1}$ is also $q$-holonomic. Then by Prop. \ref{prop:EW}, $\E_{n-1}/\mathcal I_{n-1}^1 = \E_{n-1}\otimes_{\W_{n-1}} N_{n-1}$ is $q$-holonomic.

\bigskip
\noindent\emph{Proof of Lemma \ref{lemma:box}}.

We introduce a $\Z$-grading on $\E_n$ given by degree with respect to $y_n$, with graded components  $\E_n^{(i)} = \widetilde E_{n-1}y_n^i$, where $\widetilde \E_{n-1}=k[x_1^\pm,y_1^\pm,...,x_{n-1}^\pm,y_{n-1}^\pm,x_n^\pm]/(y_ix_j-q^{\delta_{ij}}x_jy_i)$ as in \eqref{def-Etilde}.
With respect to this grading, $M^\delta_n$ may be given the structure of a graded module. Indeed,
\be M^\delta_n \simeq k[x_1^\pm,...,x_{n-1}^\pm,y_n^\pm]\,,\ee
and we take the graded components to be $M^\delta_n{}^{(i)} =k[x_1^\pm,...,x_{n-1}^\pm]y_n^i$.
The tensor product $(\E_n/\mathcal I_n)\otimes M^\delta_n$ and its submodule $\widetilde M=\E_n(v\otimes \delta^{(n)})$ inherit the $\Z$-grading from $M^\delta_n$. Explicitly, the graded components are 
\be \widetilde M^{(i)} = \widetilde \E_{n-1}\big(y_n^iv\otimes y_n^i\delta^{(n)}\big)\,.\ee

It follows that the annihilation ideal $\text{Ann}_{\E_n}(v\otimes \delta^{(n)})$ must be generated by elements that are homogeneous in $y_n$. Combined with the fact that $y_n$ is invertible, we find that $\text{Ann}_{\E_n}(v\otimes \delta^{(n)})$ can be generated entirely in degree zero, \emph{i.e.} its generators can be chosen to be elements of $\widetilde \E_{n-1}$. Moreover, we have $x_n-1\in \text{Ann}_{\E_n}(v\otimes \delta^{(n)})$, since $(x_n-1)\cdot(v\otimes \delta^{(n)}) = v\otimes (x_n-1)\delta^{(n)}=0$. All together, the annihilation ideal takes the form
\be \text{Ann}_{\E_n}(v\otimes \delta^{(n)}) = \E_n(p_1,...,p_\ell,x_n-1) \simeq \E_n(p_1|_{x_n=1},...,p_\ell|_{x_n=1},x_n-1) \label{Ann-p} \ee
for some $p_1,...,p_\ell\in \widetilde E_{n-1}$. We have used the fact that $x_n$ is central in $\widetilde E_{n-1}$  to simply set $x_1=1$ in the $p_i$'s, as indicated. This establishes a product decomposition
\be \widetilde M \simeq \E_{n-1}/(p_1|_{x_n=1},...,p_\ell|_{x_n=1}) \boxtimes \E_1/(x_n-1) = \E_{n-1}/(p_1|_{x_n=1},...,p_\ell|_{x_n=1}) \boxtimes M_1^\delta\,.\ee

It remains to show that the ideal $\E_{n-1}(p_1|_{x_n=1},...,p_\ell|_{x_n=1})$ appearing on the LHS of this product is equivalent to $\mathcal I_{n-1}^1 = (\mathcal I_n\cap \widetilde E_{n-1})|_{x_n=1}$. The following observation is key: for any $\beta\in \widetilde E_{n-1}$, we can use the $q$-commutation relations to order variables in each monomial in $\beta$ such that $x$'s are placed to the left and $y$'s are placed to the right. Then, using $y_i(v\otimes \delta_{x_n,c}) = (y_iv)\otimes (y_i \delta_{x_n,c}) = (y_iv)\otimes \delta_{x_n,c}$ for $i< n$ and $x_i(v\otimes \delta_{x_n,c}) = (x_iv)\otimes \delta_{x_n,c}$ for all $i$, we find that $\beta \cdot (v\otimes \delta^{(n)}) = (\beta v)\otimes \delta^{(n)}$ for all $\beta \in \widetilde E_{n-1}$.
More so, using $(x_n-1)\delta^{(n)}=0$ we can extend this to
\be \label{bvd} \beta \cdot (v\otimes \delta^{(n)}) = (\beta v)\otimes \delta^{(n)}= (\beta|_{x_n=1} v)\otimes \delta^{(n)} =  \beta|_{x_n=1} \cdot (v\otimes \delta^{(n)})\,. \ee

Now, if $\beta\in \mathcal I_n\cap \widetilde \E_{n-1} = \text{Ann}_{\widetilde \E_{n-1}}(v)$ then $\beta v=0$, so \eqref{bvd} implies $\beta|_{x_n=1}\in \text{Ann}_{\E_n}(v\otimes \delta^{(n)})$. From the form of the annihilation ideal \eqref{Ann-p}, we therefore have $\beta|_{x_n-1}\in \E_{n-1}(p_1|_{x_n=1},...,p_\ell|_{x_n=1})$.

Conversely, suppose that $\gamma\in \E_{n-1}(p_1|_{x_n=1},...,p_\ell|_{x_n=1})$. Then $(\gamma v)\otimes\delta^{(n)}=0$. We now observe%
\footnote{Explicitly: $(\widetilde  \E_{n-1} v)\otimes \delta^{(n)}$ is a submodule of the tensor product of modules $(\widetilde  \E_{n-1} v)\otimes (\widetilde  \E_{n-1}\delta^{(n)})$, which by definition has underlying vector space $(\widetilde  \E_{n-1} v) \underset{k[x_1^\pm,...,x_n^\pm]}\otimes (\widetilde  \E_{n-1}\delta^{(n)})$. But $\widetilde  \E_{n-1}\delta^{(n)}\simeq \C[x_1,...,x_n]/(x_n-1)$. Thus, noting that $x_n-1$ is central in $\widetilde\E_{n-1}$, the full tensor product becomes  $(\widetilde  \E_{n-1} v) \underset{k[x_1^\pm,...,x_n^\pm]}\otimes (\widetilde  \E_{n-1}\delta^{(n)}) \simeq (\widetilde  \E_{n-1} v) \underset{k[x_1^\pm,...,x_n^\pm]}\otimes \C[x_1,...,x_n]/(x_n-1) \simeq \big(\widetilde  \E_{n-1} v\big)\big/\big((x_n-1) \widetilde \E_{n-1} v\big)$. Therefore, the map $(\widetilde  \E_{n-1} v)\to(\widetilde  \E_{n-1} v)\otimes \delta^{(n)}$ has kernel contained in $(x_n-1)\widetilde E_{n-1}v$; and it is easy to check that the kernel also contains $(x_n-1)\widetilde  \E_{n-1} v$.} %
that the map $\widetilde \E_{n-1} v\to (\widetilde \E_{n-1} v)\otimes \delta^{(n)}$ of left $\widetilde E_{n-1}$-modules has kernel $(x_n-1)\widetilde E_{n-1}v$. Therefore, $(\gamma v)\otimes\delta^{(n)}=0$ implies that there exists $\tilde \gamma\in \widetilde \E_{n-1}$ such that $\gamma v=(x_n-1)\tilde\gamma v$; or equivalently that there exists $\hat\gamma\in \widetilde \E_{n-1}$ such that $\hat\gamma v=0$ and $\hat\gamma|_{x_n=1}=\gamma$ (just set $\hat\gamma = \gamma-(x_n-1)\tilde \gamma$). Since $\hat\gamma \in \text{Ann}_{\widetilde \E_{n-1}}(v) =  \mathcal I_n\cap \widetilde \E_{n-1}$, it follows that $\gamma \in (\mathcal I_n\cap \widetilde \E_{n-1})\big|_{x_n=1} = \mathcal I_{n-1}^1$.
\phantom{x} \hfill $\square$ \bigskip

\section{Further examples and computations}
\label{app:comp}

The Jones polynomials for the zero-framed trefoil ($\mathbf{3_1}$) and $\mathbf{5_2}$ knots are readily computed using a general formula for $p$-twist knots \cite{Masbaum, Habiro-simple} (see also \cite{GS-twist}):
\be J_n^p(q) := \sum_{k=0}^n \sum_{j=0}^k (-1)^{j+1} q^{k+pj(j+1)+\frac12j(j-1)} \frac{(q^{2j+1}-1)(q;q)_k(q^{1-n};q)_k(q^{1+n};q)_k}{(q;q)_{k+j+1}(q;q)_{k-j}}\\,. \ee
In our normalizations and choices of chirality, we have
\be J_N^{\mathbf{3_1}}(q) = \frac{q^N-q^{-N}}{q-q^{-1}} J_N^{p=1}(q^{-2})\,,\qquad J_N^{\mathbf{5_2}}(q) = \frac{q^N-q^{-N}}{q-q^{-1}} J_N^{p=2}(q^{-2})\,. \ee

We computed ADO invariants directly, using the $(1,1)$-tangle diagrams in Figure \ref{fig:diags}, and then changing the framing from blackboard to zero framing. We performed computations for $2\leq r\leq 11$. For convenience, we introduce the normalization
\be \hat N_K^r(\alpha) := i^{1-r}(\zeta_{2r}^\alpha-\zeta_{2r}^{-\alpha})N_K^r(\alpha-1)\,. \ee

 \begin{figure}[htb]
\centering
$$\raisebox{.5in}{\includegraphics[width=.9in]{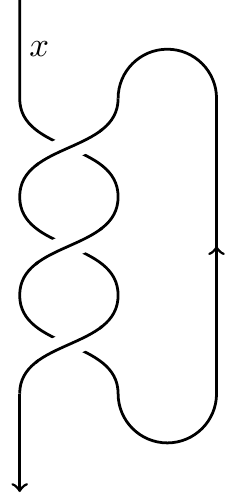}}\hspace{1in} \includegraphics[width=1.3in]{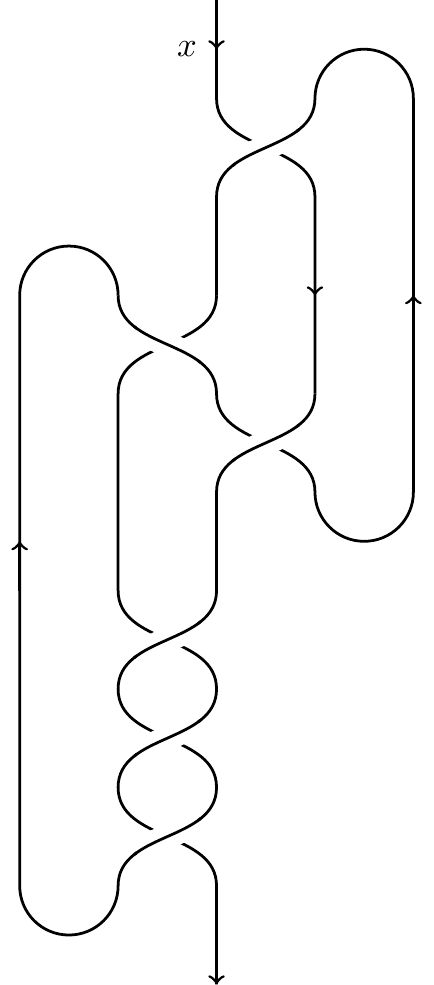} $$
\caption{Tangle diagrams whose closures are the trefoil (left) and $\mathbf{5_2}$ knot (right).}
\label{fig:diags}
\end{figure}

\noindent Letting $X^{(n)}:=x^n-x^{-n}$, the ADO invariants for the trefoil and $\mathbf{5_2}$ knots are:
$${\small \begin{array}{c|l}
 r & \hat N_{\mathbf{3_1}}^r(\alpha) \\\hline
 2 & -X^{(3)}  \\
 3 & q^2X^{(5)} + qX^{(1)} \\
 4 & q^2 X^{(7)}+X^{(3)}+q^2 X^{(1)} \\
 5 & q^2X^{(9)} -q^4 X^{(5)}+qX^{(3)} \\
 6 & q^2 X^{(11)} -q^4 X^{(7)}+X^{(5)}+X^{(1)} \\
 7 & q^2 X^{(13)} -q^4 X^{(9)}-q^6 X^{(7)} -q^5 X^{(3)}+q^2 X^{(1)} \\
 8 & q^2 X^{(15)} -q^4 X^{(11)}-q^6 X^{(9)} - q^4 X^{(5)} + X^{(3)}\\
  9 & q^2 X^{(17)}-q^4 X^{(13)}-q^6 X^{(11)}-q^3 X^{(7)}-q^7 X^{(5)} -q^8 X^{(1)} \\
 10 & q^2 X^{(19)}-q^4 X^{(15)}-q^6 X^{(13)}-q^2 X^{(9)}-q^6 X^{(7)} - q^6 X^{(3)}+q^2 X^{(1)} \\
 11 & q^2 X^{(21)}-q^4 X^{(17)} -q^6 X^{(15)}-qX^{(11)}-q^5 X^{(9)} - q^4 X^{(5)}-q^{10}X^{(3)}
 \end{array} } \quad (q=\zeta_{2r},\,x=\zeta_{2r}^\alpha\;\text{for each $r$}) $$
 
$$ {\small \hspace{-.4in} \begin{array}{c|l}
r &  \hat N_{\mathbf{5_2}}^r(\alpha) \\[.1cm]\hline
2 & -2X^{(3)}-X^{(1)} \\[.1cm]\hline
3 & (2q^2-1)X^{(5)}+2q^2 X^{(3)}+2q^2 X^{(1)}\\[.1cm]\hline
4 & (2q^2-2)X^{(7)}+(3q^2-q)X^{(5)}+(3q^2-1)X^{(3)}+(2q^2-1)X^{(1)} \\[.1cm]\hline
5 & (2q^2-q-2)X^{(9)}+(2q^3+2q^2-2)X^{(7)}+(2q^3+2q^2+q-3)X^{(5)}+(2q^3+q^2+q-2)X^{(3)} \\& +(q^3+q^2-2)X^{(1)} \\[.1cm]\hline
6 & -(4q^4+2)X^{(9)} -(6q^4+2)X^{(7)}-(6q^4+1)X^{(5)}-(4q^4+2)X^{(3)}-2X^{(1)} \\[.1cm]\hline
7 & -(q^4+2q^3-2q^2+1)X^{(13)} +(4q^5-2q^4-4)X^{(11)} +(5q^5-2q^4+2q^3-7)X^{(9)} \\ & +(6q^5-q^4+3q^3-2q^2+2q-7)X^{(7)}+(5q^5-2q^4+3q^3-q^2+q-7)X^{(5)} \\& +(3q^5-2q^4+q^3-q^2-4)X^{(3)} -(q^4+q^3-q^2+2)X^{(1)} \\[.1cm]\hline
8 & -(2q^6+2q^4-2q^2+2)X^{(15)}+(q^6-3q^4-q^2-5)X^{(13)}+(3q^6-q^4-3q^2-9)X^{(11)}+(7q^6-3q^2-10)X^{(9)} \\&+(7q^6-3q^2-10)X^{(7)}+(4q^6-q^4-3q^2-8)X^{(5)}+(q^6-3q^4-2q^2-4)X^{(3)} -(q^6+2q^4-q^2+1)X^{(1)} \\[.1cm]\hline
9 & -(4q^5-4q^2+1)X^{(17)} - (4q^5+4q^3-2q^1+4q)X^{(15)} - (2q^5+q^4+5q^3+7q+5)X^{(13)} \\& + (q^5+2q^4-3q^3-6q^2-11q-8)X^{(11)} +(3q^5-q^3-7q^2-10q-12)X^{(9)} \\& +(2q^5+q^4-2q^3-7q^2-10q-8)X^{(7)} -(q^5+2q^4+4q^3+2q^2+6q+5)X^{(5)} \\& - (3q^5+2q^4+3q^3-q^2+2q)X^{(3)} - (3q^5+q^4-3q^2+1)X^{(1)} \\[.1cm]\hline
10 & -(q^6-4q^2)X^{(19)} -(6q^6+4q^4-2)X^{(17)}-(8q^6+10q^4+4q^2+4)X^{(15)} -(6q^6+14q^4+10q^2+14)X^{(13)} \\& - (22q^4+8q^2+22)X^{(11)} + (q^6-22q^4-8q^2-22)X^{(9)}  - (6q^6+14q^4+10q^2+14)X^{(7)} \\& - (9q^6+8q^4+6q^2+3)X^{(5)} -(7q^6+2q^4+q^2-2)X^{(3)} - (3q^6-q^4-2q^2-1)X^{(1)} \\[.1cm]\hline
11 &-(2q^7-q^5-2q^4-2q^2-2q+2)X^{(21)} -(2q^9+6q^7+4q^5-4q^4+4q^3-6q^2-2)X^{(19)} \\& - (3q^9+5q^8+7q^7+3q^6+7q^5+7q^3-4q^2+1)X^{(17)} \\& -(2q^9+5q^8+10q^7+6q^6+9q^5+5q^4+12q^3+2q^2+5q+3)X^{(15)}  \\&
+(3q^9-9q^8-6q^7-11q^6-8q^5-14q^4-10q^3-9q^2-3q-12)X^{(13)} \\& +(6q^9-8q^8-4q^7-15q^6-6q^5-18 q^4-9q^3-15q^2-2q-13)X^{(11)} \\& + (3q^9-9q^8-5q^7-12q^6-7q^5-15q^4-9q^3-10q^2-3q-12)X^{(9)} \\&  -(2q^9+5q^8+9q^7+7q^6+8q^5+7q^4+10q^3+3q^2+4q+3)X^{(7)} \\& - (4q^9+4q^8+7q^7+3q^6+6q^5+2q^4+6q^3-3q^2)X^{(5)} \\& -(2q^9+5q^7+3q^5-2q^4+3q^3-5q^2-3)X^{(3)} -(2q^7-q^5-q^4-2q^2-2q+1)X^{(1)}
\end{array} }$$

Inhomogeneous recursion relations for the colored Jones polynomials of $\mathbf{3_1}$ and $\mathbf{5_2}$ were found in \cite{GL-J,GS-twist}; in the current normalization, the recursions take the form
\be \label{app-inhom} (q-q^{-1})A_{\mathbf{3_1}}(x,y;q)J_N^{\mathbf{3_1}}(q) =B_{\mathbf{3_1}}(q^N;q)\,,\qquad (q-q^{-1})A_{\mathbf{5_2}}(x,y;q)J_N^{\mathbf{5_2}}(q) = B_{\mathbf{5_2}}(q^N;q)\,, \ee
with
\begin{align}
 A_{\mathbf{3_1}}( x, y;q) &= q^3 x^6 y - 1\,, \\[.1cm]
 B_{\mathbf{3_1}}( x;q) &= q^2 x(q^2x^4-1)\,,\\[.1cm]
  A_{\mathbf{5_2}}( x, y;q) &= -q^{28}(1-q^2 x^4)(1-q^4 x^4) x^{14} y^3 \notag\\&\hspace{-45pt}
 -q^{5}(1-q^2 x^4)(1-q^8 x^4) x^4(1-q^4 x^2-q^4(1-q^2)(1-q^4) x^4+q^8(1+q^6) x^6+2q^{14} x^8-q^{18} x^{10}) y^2 \notag \\&\hspace{-45pt}
 +(1-q^4 x^4)(1-q^{10} x^4)(1-2q^2 x^2-q^2(1+q^6) x^4+q^4(1-q^2)(1-q^4) x^6+q^{10} x^8-q^{12} x^{10}) y \notag \\
  & \quad  -q(1-q^8 x^4)(1-q^{10} x^4)\,, \\
  B_{\mathbf{5_2}}(x;q) &=  q^5 x^3+q^7(1+q^2)x^5-q^7(1+q^8)x^7+\frac{q^6-q^{-6}}{q-q^{-1}}(-q^{14}x^9+q^{20}x^{13})\notag \\
  &\hspace{2in} -q^{19}(1+q^8)x^{15}-q^{25}(1+q^2)x^{17}-q^{29}x^{19}\,.
\end{align}
These imply homogeneous recursions
\be \label{app-hom} \widetilde A_K(x,y;q)J_N^K(q) :=    \big[B_K(x;q)y - B_K(qx;q)\big]A_K(x,y;q) J_N^K(q) = 0 \qquad (K=\mathbf{3_1},\,\mathbf{5_2})\,, \ee
just as in the figure-eight example \eqref{41-hom} in the Introduction.

We checked explicitly for each $2\leq r\leq 11$ that the ADO invariants satisfy inhomogeneous recursions
\be \label{31-inh} \begin{array}{l} 
 A_{\mathbf{3_1}}(x,y;\zeta_{2r})\hat N_{\mathbf{3_1}}^r(\alpha) = (\zeta_{2r}^{2r\alpha}-1+\zeta_{2r}^{-2r\alpha})B_{\mathbf{3_1}}(\zeta_{2r}^\alpha,\zeta_{2r}) \\[.2cm]
 A_{\mathbf{5_2}}(x,y;\zeta_{2r})\hat N_{\mathbf{5_2}}^r(\alpha) = (2\zeta_{2r}^{2r\alpha}-3+2\zeta_{2r}^{-2r\alpha})B_{\mathbf{5_2}}(\zeta_{2r}^\alpha,\zeta_{2r})  \end{array}
\ee
with exactly the same $A$ and $B$ polynomials. Again, these imply homogeneous recursions
\be \label{31-hom}  \widetilde A_K(x,y;\zeta_{2r}) \hat N_K^r(\alpha) = 0  \qquad r\in \N_{\geq 2}  \qquad (K=\mathbf{3_1},\,\mathbf{5_2})\,;\ee
with the same $\widetilde A_K(x,y;q) =   \big[B_K(x;q)y - B_K(qx;q)\big]A_K(x,y;q)$.
Note that the prefactors $(x^{2r}-1+x^{-2r})$ and $(2x^{2r}-3+2x^{-2r})$ appearing in \eqref{31-inh} may be factored out from the homogeneous recursion \eqref{31-hom}, since they commute with $y$ and just behave like overall constants.

In terms of the standard normalization of the ADO invariant used in the main body of the paper ($N_K$ rather than $\hat N_K$), the homogeneous recursions take the form
\be \widetilde A_K(qx,-y;\zeta_{2r}) N_K^r(\alpha) = 0  \qquad r\in \N_{\geq 2}  \qquad (K=\mathbf{3_1},\,\mathbf{5_2})\,, \ee
in perfect agreement with Theorem \ref{thm:Jones}.

\bibliography{ADO-invariant-recursion}

\newcommand{\etalchar}[1]{$^{#1}$}
\providecommand{\bysame}{\leavevmode\hbox to3em{\hrulefill}\thinspace}
\providecommand{\MR}{\relax\ifhmode\unskip\space\fi MR }
\providecommand{\MRhref}[2]{%
  \href{http://www.ams.org/mathscinet-getitem?mr=#1}{#2}
}
\providecommand{\href}[2]{#2}
\begin{thebibliography}{CGPM15b}

\bibitem[ADO92]{akutsu1992invariants}
Yasuhiro Akutsu, Testuo Deguchi, and Tomotada Ohtsuki, \emph{Invariants of
  colored links}, Journal of Knot Theory and its Ramifications \textbf{1}
  (1992), no.~02, 161--184.

\bibitem[BB]{Beliakova}
Anna Beliakova and Christian Blanchet, \emph{Non-semisimple quantum
  invariants}, Talk at UC Davis, Jan 21, 2020.

\bibitem[Ber71]{Bernstein}
Joseph Bernstein, \emph{Modules over a ring of differential operators. {A}n
  investigation of the fundamental solutions of equations with constant
  coefficients}, Funkcional. Anal. i Prilo\v{z}en. \textbf{5} (1971), no.~2,
  1--16. \MR{0290097}

\bibitem[CBEG07]{CBEG}
William Crawley-Boevey, Pavel Etingof, and Victor Ginzburg,
  \emph{Noncommutative geometry and quiver algebras}, Adv. Math. \textbf{209}
  (2007), no.~1, 274--336. \MR{2294224}

\bibitem[CCG{\etalchar{+}}94]{Apoly}
Daryl Cooper, Marc Culler, Henry Gillet, Daryl Long, and Peter Shalen,
  \emph{Plane curves associated to character varieties of {$3$}-manifolds},
  Invent. Math. \textbf{118} (1994), no.~1, 47--84. \MR{1288467 (95g:57029)}

\bibitem[CGPM15a]{costantino2015quantum}
Francesco Costantino, Nathan Geer, and Bertrand Patureau-Mirand, \emph{Quantum
  invariants of 3-manifolds via link surgery presentations and non-semi-simple
  categories}, preprint (2015), ar{X}iv:1202.355v3.

\bibitem[CGPM15b]{CGP-relations}
\bysame, \emph{Relations between {W}itten-{R}eshetikhin-{T}uraev and
  nonsemisimple $sl_2$ 3-manifold invariants}, Algebr. Geom. Topol. \textbf{15}
  (2015), no.~3, 1363--1386. \MR{3361139}

\bibitem[CGPM15c]{costantino2015some}
\bysame, \emph{Some remarks on the unrolled quantum group of sl (2)}, Journal
  of Pure and Applied Algebra \textbf{219} (2015), no.~8, 3238--3262.

\bibitem[DGLZ09]{DGLZ}
Tudor Dimofte, Sergei Gukov, Jonatan Lenells, and Don Zagier, \emph{Exact
  results for perturbative {C}hern-{S}imons theory with complex gauge group},
  Commun. Number Theory Phys. \textbf{3} (2009), no.~2, 363--443. \MR{2551896}

\bibitem[Dim13]{Dimofte-QRS}
Tudor Dimofte, \emph{Quantum {R}iemann surfaces in {C}hern-{S}imons theory},
  Adv. Theor. Math. Phys. \textbf{17} (2013), no.~3, 479--599. \MR{3250765}

\bibitem[Dim15]{Dimofte-k}
\bysame, \emph{Complex {C}hern-{S}imons theory at level {$k$} via the 3d-3d
  correspondence}, Comm. Math. Phys. \textbf{339} (2015), no.~2, 619--662.
  \MR{3370614}

\bibitem[EG02]{EtingofGinzburg}
Pavel Etingof and Victor Ginzburg, \emph{Symplectic reflection algebras,
  {C}alogero-{M}oser space, and deformed {H}arish-{C}handra homomorphism},
  Invent. Math. \textbf{147} (2002), no.~2, 243--348. \MR{1881922}

\bibitem[FK94]{FaddeevKashaev}
Ludwig Faddeev and Rinat Kashaev, \emph{Quantum dilogarithm}, Modern Phys.
  Lett. A \textbf{9} (1994), no.~5, 427--434. \MR{1264393}

\bibitem[Gar04]{Gar-AJ}
Stavros Garoufalidis, \emph{On the characteristic and deformation varieties of
  a knot}, Proceedings of the {C}asson {F}est, Geom. Topol. Monogr., vol.~7,
  Geom. Topol. Publ., Coventry, 2004, pp.~291--309. \MR{2172488}

\bibitem[GHN{\etalchar{+}}20]{Gukov:2020lqm}
Sergei Gukov, Po-Shen Hsin, Hiraku Nakajima, Sunghyuk Park, Du~Pei, and Nikita
  Sopenko, \emph{Rozansky-{W}itten geometry of {C}oulomb branches and
  logarithmic knot invariants}, preprint (2020), ar{X}iv:2005.05347.

\bibitem[GK12]{GaroufalidisKoutschan}
Stavros Garoufalidis and Christoph Koutschan, \emph{The noncommutative
  {$A$}-polynomial of {$(-2,3,n)$} pretzel knots}, Exp. Math. \textbf{21}
  (2012), no.~3, 241--251. \MR{2988577}

\bibitem[GL05]{GL-J}
Stavros Garoufalidis and Thang T.~Q. L{\^e}, \emph{The colored jones function
  is q-holonomic}, Geometry \& Topology \textbf{9} (2005), no.~3, 1253--1293.

\bibitem[GL16]{GL-survey}
Stavros Garoufalidis and Thang T.~Q. L\^{e}, \emph{A survey of {$q$}-holonomic
  functions}, Enseign. Math. \textbf{62} (2016), no.~3-4, 501--525.
  \MR{3692896}

\bibitem[GM19]{GukovManolescu}
Sergei Gukov and Ciprian Manolescu, \emph{A two-variable series for knot
  complements}, preprint (2019), ar{X}iv:1904.06057.

\bibitem[GPM18]{geer2018trace}
Nathan Geer and Bertrand Patureau-Mirand, \emph{The trace on projective
  representations of quantum groups}, Letters in Mathematical Physics
  \textbf{108} (2018), no.~1, 117--140.

\bibitem[GPMT09]{geer2009modified}
Nathan Geer, Bertrand Patureau-Mirand, and Vladimir Turaev, \emph{Modified
  quantum dimensions and re-normalized link invariants}, Compositio Mathematica
  \textbf{145} (2009), no.~1, 196--212.

\bibitem[GPPV20]{GPPV}
Sergei Gukov, Du~Pei, Pavel Putrov, and Cumrun Vafa, \emph{B{PS} spectra and
  3-manifold invariants}, J. Knot Theory Ramifications \textbf{29} (2020),
  no.~2, 2040003, 85. \MR{4089709}

\bibitem[GPV17]{GPV}
Sergei Gukov, Pavel Putrov, and Cumrun Vafa, \emph{Fivebranes and 3-manifold
  homology}, J. High Energy Phys. (2017), no.~7, 071, front matter+80.
  \MR{3686727}

\bibitem[GS10]{GS-twist}
Stavros Garoufalidis and Xinyu Sun, \emph{The non-commutative {$A$}-polynomial
  of twist knots}, J. Knot Theory Ramifications \textbf{19} (2010), no.~12,
  1571--1595. \MR{2755491}

\bibitem[Guk05]{Gukov}
Sergei Gukov, \emph{Three-dimensional quantum gravity, {C}hern-{S}imons theory,
  and the {A}-polynomial}, Comm. Math. Phys. \textbf{255} (2005), no.~3,
  577--627. \MR{2134725}

\bibitem[Hab00]{Habiro-simple}
Kazuo Habiro, \emph{On the colored {J}ones polynomials of some simple links},
  no. 1172, 2000, Recent progress towards the volume conjecture (Japanese)
  (Kyoto, 2000), pp.~34--43. \MR{1805727}

\bibitem[Hab04]{Habiro-cyc}
\bysame, \emph{Cyclotomic completions of polynomial rings}, Publ. Res. Inst.
  Math. Sci. \textbf{40} (2004), no.~4, 1127--1146. \MR{2105705}

\bibitem[Hab07]{Habiro-sl2}
\bysame, \emph{An integral form of the quantized enveloping algebra of {${\rm
  sl}_2$} and its completions}, J. Pure Appl. Algebra \textbf{211} (2007),
  no.~1, 265--292. \MR{2333771}

\bibitem[Jor14]{Jordan}
David Jordan, \emph{Quantized multiplicative quiver varieties}, Adv. Math.
  \textbf{250} (2014), 420--466. \MR{3122173}

\bibitem[Kas95]{kassel1995quantum}
Christian Kassel, \emph{Quantum groups}, Graduate texts in mathematics,
  Springer-Verlag, 1995.

\bibitem[Kas97]{Kashaev}
Rinat Kashaev, \emph{The hyperbolic volume of knots from the quantum
  dilogarithm}, Lett. Math. Phys. \textbf{39} (1997), no.~3, 269--275.
  \MR{1434238}

\bibitem[Kas77]{Kashiwara-Bfunctions}
Masaki Kashiwara, \emph{{$B$}-functions and holonomic systems. {R}ationality of
  roots of {$B$}-functions}, Invent. Math. \textbf{38} (1976/77), no.~1,
  33--53. \MR{430304}

\bibitem[KS12]{KashiwaraSchapira-DQ}
Masaki Kashiwara and Pierre Schapira, \emph{Deformation quantization modules},
  Ast\'{e}risque (2012), no.~345, xii+147. \MR{3012169}

\bibitem[Los12]{Losev}
Ivan Losev, \emph{Isomorphisms of quantizations via quantization of
  resolutions}, Adv. Math. \textbf{231} (2012), no.~3-4, 1216--1270.
  \MR{2964603}

\bibitem[M{\etalchar{+}}08]{murakami2008colored}
Jun Murakami et~al., \emph{Colored alexander invariants and cone-manifolds},
  Osaka Journal of Mathematics \textbf{45} (2008), no.~2, 541--564.

\bibitem[Mas03]{Masbaum}
Gregor Masbaum, \emph{Skein-theoretical derivation of some formulas of
  {H}abiro}, Algebr. Geom. Topol. \textbf{3} (2003), 537--556. \MR{1997328}

\bibitem[MM01]{MurakamiMurakami}
Hitoshi Murakami and Jun Murakami, \emph{The colored {J}ones polynomials and
  the simplicial volume of a knot}, Acta Math. \textbf{186} (2001), no.~1,
  85--104. \MR{1828373}

\bibitem[Mur08]{Murakami-cone}
Jun Murakami, \emph{Colored {A}lexander invariants and cone-manifolds}, Osaka
  J. Math. \textbf{45} (2008), no.~2, 541--564. \MR{2441954}

\bibitem[Oht02]{ohtsuki2002quantum}
Tomotada Ohtsuki, \emph{Quantum invariants: a study of knots, 3-manifolds, and
  their sets}, vol.~29, World Scientific, 2002.

\bibitem[RT90]{reshetikhin1990ribbon}
Nicolai~Yu Reshetikhin and Vladimir Turaev, \emph{Ribbon graphs and their
  invariants derived from quantum groups}, Communications in Mathematical
  Physics \textbf{127} (1990).

\bibitem[RT91]{RT}
Nicolai Reshetikhin and Vladimir Turaev, \emph{Invariants of {$3$}-manifolds
  via link polynomials and quantum groups}, Invent. Math. \textbf{103} (1991),
  no.~3, 547--597. \MR{1091619}

\bibitem[Sab93]{Sabbah}
Claude Sabbah, \emph{Syst\`emes holonomes d'\`equations aux q-diff\`erences},
  D-modules and mircolocal geometry (M.~Kashiwara, T.~Monteiro-Fernandes, and
  P.~Schapira, eds.), de Gruyter, 1993, pp.~125--147.

\bibitem[Tur88]{Turaev}
Vladimir Turaev, \emph{The {Y}ang-{B}axter equation and invariants of links},
  Invent. Math. \textbf{92} (1988), no.~3, 527--553. \MR{939474}

\bibitem[Tur16]{turaev2016quantum}
\bysame, \emph{Quantum invariants of knots and 3-manifolds}, vol.~18, Walter de
  Gruyter GmbH \& Co KG, 2016.

\bibitem[Wil20]{Willetts}
Sonny Willetts, \emph{A unification of the ado and colored jones polynomials of
  a knot}, preprint (2020), ar{X}iv:2003.09854.

\bibitem[Wit89]{Witten-Jones}
Edward Witten, \emph{Quantum field theory and the {J}ones polynomial}, Comm.
  Math. Phys. \textbf{121} (1989), no.~3, 351--399. \MR{990772}

\bibitem[WZ92]{wilf1992algorithmic}
Herbert Wilf and Doron Zeilberger, \emph{An algorithmic proof theory for
  hypergeometric (ordinary and “q”) multisum/integral identities},
  Inventiones mathematicae \textbf{108} (1992), no.~1, 575--633.

\bibitem[Zei90]{zeilberger1990holonomic}
Doron Zeilberger, \emph{A holonomic systems approach to special functions
  identities}, Journal of computational and applied mathematics \textbf{32}
  (1990), no.~3, 321--368.

\end{thebibliography}
	
\end{document}